\documentclass[12pt,twoside,a4paper]{amsart}

\title{Representation of Small  Integers by Binary Forms}
\author{Shabnam Akhtari}
\address{Department of Mathematics\\
Fenton Hall\\
University of Oregon\\
Eugene, OR 97403-1222 USA}
 \email {akhtari@uoregon.edu}
 

\begin{document}

\newtheorem{thm}{Theorem}[section]
\newtheorem{prop}[thm]{Proposition}
\newtheorem{lemma}[thm]{Lemma}
\newtheorem{cor}[thm]{Corollary}
\newtheorem{conj}[thm]{Conjecture}

\begin{abstract}
We establish some upper bounds for the number of integer solutions to the Thue inequality $|F(x , y)| \leq m$, where $F$ is a binary form of degree $n \geq 3$ and with non-zero discriminant $D$, and $m$ is an integer.  Our upper bounds are independent of $m$, when $m$ is smaller than $|D|^{\frac{1}{4(n-1)}}$. We also consider the Thue equation $|F(x , y)| = m$ and give some upper bounds for the number of its integral solutions. In the case of equation, our upper bounds will be independent of integer  $m$, when $ m < |D|^{\frac{1}{2(n-1)}}$.
\end{abstract}
\maketitle

\section{Introduction and statements of the results}\label{Intro}


Let $F(x , y)$ be a binary form with integral coefficients.  Let $m$ be a positive integer.
In this manuscript, we  study  the Thue inequality
\begin{equation}\label{Thueineq}
|F(x , y)| \leq m
\end{equation}
and the Thue equation
\begin{equation}\label{Thueequ}
|F(x , y)| = m.
\end{equation}
It is well-known \cite{Thu} that such inequalities and equations have at most finitely many solutions in integers $x$ and $y$, provided that $n \geq 3$ and the maximal number of pairwise non-proportional linear forms over $\mathbb{C}$ dividing $F$ is at least $3$. \newline
\textbf{Definition}. A pair $(x , y) \in \mathbb{Z}^2$ is called a primitive solution to inequality \eqref{Thueineq} (or to equation  \eqref{Thueequ}) if it satisfies the inequality (or the equation) and  $\gcd(x , y) = 1$.

For any nonzero integer $m$ let $\omega(m)$ denote the number of distinct prime factors of $m$. In 1933, Mahler \cite{Mah23} proved that if $F$ is irreducible then equation \eqref{Thueequ} has at most $C_{1}^{1 + \omega(m)}$ solutions in  integers $x$ and $y$, where $C_{1}$ is a positive number that depends on $F$ only.
In 1987, Bombieri and Schmidt \cite{Bom} showed that the number of solutions of $F(x , y) = m$ in co-prime integers $x$ and $y$ is at most 
$$
C_{2} \,  n^{1 + \omega(m)},
$$
where $C_{2}$ is an absolute constant.  The latter  upper bound is independent of the coefficients of the form $F$; a result of this flavor was first deduced in 1983 by Evertse \cite{Eve}.

As for the inequality, Mahler \cite{Mah9} showed that (\ref{Thueineq}) has at most $c(F) m^{2/n}$ solutions, where $c(F)$ depends only of $F$.  This bound is what one would expect intuitively. Thunder \cite{Thun} showed that $c(F)$ can be replaced by a constant depending only on $n$. Our main goal here is to establish an upper bound, independent of $m$, for the number of solutions to  \eqref{Thueineq} and \eqref{Thueequ} when $m$ is small enough in terms of the discriminant of $F$. Throughout this article, we regard $(x , y)$ and $(-x , -y)$ as one solution.

\begin{thm}\label{mainineq}
 Let $F(x , y) \in \mathbb{Z}[x , y]$ be an irreducible binary form with degree $n\geq 3$ and discriminant $D$. Let $m$ be an integer with
 $$
 0 < m  \leq \frac{|D|^{\frac{1}{4(n-1)} - \epsilon} }  {(3.5)^{n/2} n^{ \frac{n}{4(n-1) } } },
 $$
 where $ 0< \epsilon < \frac{1}{4(n-1)}$.
 Then the inequality $0< |F(x , y)| \leq m$ has at most
  \[ \left\{
  \begin{array}{l l}
  7n + \frac{n}{2(n-1) \epsilon} & \quad \text{if} \, \, n\geq 5 \\
  9n + \frac{n}{2(n-1) \epsilon} & \quad \text{if} \, \, n= 3, 4
  \end{array}
\right.  \]
    primitive solutions. In addition to the above assumptions, if we assume that the polynomial $F(x , 1)$ has $2q$ non-real roots then the number of primitive solutions does not exceed
  \[ \left\{
  \begin{array}{l l}
 7n -12 q + \frac{n-q}{2(n-1) \epsilon} & \quad \text{if} \, \, n\geq 5 \\
  9n -16 q + \frac{n-q}{2(n-1) \epsilon} &  \quad \text{if} \, \, n= 3, 4.
  \end{array}
\right. \]
  \end{thm}
  
  The Thue inequality
\begin{equation}\label{1.2ih}
0 < |F(x , y)| \leq m
\end{equation} 
has been studied by Evertse and Gy\H{o}ry (see \cite{Eve} and \cite{EG16}).
Define, for $ 3 \leq n < 400$,
$$
\left( N(n) , \delta(n) \right) = 
\left( 6 n 7^{\binom{n}{3}},\,  \frac{5}{6} n (n-1) \right) $$
and for $n > 400$,
$$
\left( N(n) , \delta(n) \right) = \left( 6n , \, 120(n -1)\right). 
$$
They prove that if
$$
|D| > m^{\delta(n)} \exp (80 n (n -1)),
$$
then the number of solutions to (\ref{1.2ih}) in co-prime integers $x$ and $y$ is at most $N(n)$.
Gy\H{o}ry has \cite{Gyo1} also proved, for binary form $F$ of degree $n \geq 3$, that if $0 <a < 1$ and 
$$
\left| D \right| \geq n^n (3.5^{n} m^2)^{\left(2(n-1)/(1-a)\right)},
$$
then
the number of primitive  solutions to \eqref{Thueequ} is
at most $25n + (n+2) \left(\frac{2}{a} + \frac{1}{4}\right)$. Furthermore, if $F$ is reducible then the number of primitive  solutions to \eqref{Thueequ} is at most $5n + (n+2) \left(\frac{2}{a} + \frac{1}{4}\right)$.

 Our next theorem is inspired by a paper of Stewart's \cite{Ste}, where he shows, among other things,  that  if $\epsilon > 0$, the discriminant $D$ of $F$ is non-zero, and 
$$
|D|^{1/n(n-1)} \geq m^{\frac{2}{n+\epsilon}},
$$
then the number of pairs of co-prime integers $(x , y)$ for which $F(x , y) = m$ holds is at most 
$$
1400 \left( 1 + \frac{1}{8 \epsilon n}\right) n.
$$

 \begin{thm}\label{maineq}
 Let $F(x , y) \in \mathbb{Z}[x , y]$ be an irreducible binary form of degree $n\geq 3$ and discriminant $D$.  Let $m$ be an integer with
  $$
0 < m  \leq \frac{|D|^{\frac{1}{2(n-1)} - \epsilon} }  {(3.5)^{n/2} n^{ \frac{n}{2(n-1) } } },
 $$
 where $ 0< \epsilon < \frac{1}{2(n-1)}$.
 Then the equation $|F(x , y)| = m$ has at most
  \[ \left\{
  \begin{array}{l l}
  7n + \frac{n}{(n-1) \epsilon} & \quad \text{if}\, \,  n \geq 5\\
  9n + \frac{n}{(n-1) \epsilon} & \quad \text{if} \, \, n = 3, 4 
  \end{array} 
 \right.  \]
   primitive solutions. In addition to the above assumptions, if we assume that the polynomial $F(X , 1)$ has $2q$ non-real roots then the number of primitive solutions does not exceed
  \[ \left\{
  \begin{array}{l l}
 7n -12 q + \frac{n-q}{(n-1) \epsilon}& \quad \text{if}\, \,  n \geq 5\\
 9n -16 q + \frac{n-q}{(n-1) \epsilon}& \quad \text{if}\, \,  n =3, 4.
 \end{array} 
 \right.  \]
  \end{thm}

 It turns out that the assumption that $F$ be irreducible in the above theorems is not really necessary. As long as our binary form is of degree at least $3$ and  is not a power of a linear or quadratic form, the statements of Theorems \ref{mainineq} and \ref{maineq} hold. Indeed, we have
  \begin{thm}\label{mainred}
 Let $F(x , y) \in \mathbb{Z}[x , y]$ be a reducible binary form of degree $n\geq 3$.  Assume that either the maximal number of pairwise non-proportional linear forms over $\mathbb{C}$ dividing $F$ is at least $3$ or $F$ is a product of powers of $2$ linear forms with rational coefficients. Let $m$ be an integer with
$$
 0 < m \leq  \frac{|D|^{\frac{1}{4(n-1)} - \epsilon} }  {(3.5)^{n/2} n^{ \frac{n}{4(n-1) } } },
 $$
 where $ 0< \epsilon < \frac{1}{4(n-1)}$.
Then the inequality $|F(x , y)| \leq m$ has at most 
$$2 n + \frac{n }{ 2(n-1) \epsilon }
$$ primitive solutions.
  \end{thm}

One can  use the above results to give upper bounds for the number of solutions of Thue inequalities and equations without assuming $m$ is small. In these cases, however, our bounds will depend on $m$.

\begin{thm}\label{THUNineq}
Let $F(x , y) \in \mathbb{Z}[x , y]$ be an irreducible  binary form of degree $n\geq 3$ and $m$ a positive integer. The inequality $|F(x , y)| \leq m$ has at most 
$$
\left(p_{1}(m , D) + 1 \right) \left( 9n + \frac{4}{n-1}  \right) 
$$ 
primitive solutions, where $D =D(F)$ is the discriminant of $F$ and $p_{1}(m , D)$ is the smallest prime number $p$ that satisfies 
$$
p \geq m^{\frac{8}{n}} (3.5)^4 n^{\frac{2}{n-1}}.
$$
\end{thm}

\begin{thm}\label{THUNequ}
Let $F(x , y) \in \mathbb{Z}[x , y]$ be an irreducible  binary form of degree $n\geq 3$ and $m$ a positive integer. The equation $|F(x , y)| = m$ has at most 
$$
\left(p_{2}(m , D) + 1 \right) \left( 9 n + \frac{4}{n-1}  \right) 
$$ 
primitive solutions, where $D =D(F)$ is the discriminant of $F$ and $p_{2}(m , D)$ is the smallest prime number $p$ that satisfies 
$$
p \geq m^{\frac{4}{n}} (3.5)^2 n^{\frac{2}{n-1}}.
$$
\end{thm}

\textbf{Remark}.  In the statement of Theorems  \ref{THUNineq} and \ref{THUNequ}, upper bounds are given in terms of the smallest prime number greater than or equal to a parameter. For any $N> 2$ there is always a prime $p$ such that $N \leq p < 2N$. Therefore twice the parameter could be used in place of the prime number. This would have the advantage of making the bound more explicit.

Birch and Merriman \cite{BirM} proved that for arbitrary $n \geq 4 $, there are only finitely many equivalent classes of binary forms in $\mathbb{Z}[x , y]$  of degree $n$ and a fixed discriminant. Evertse and Gy\H{o}ry \cite{EG1} proved an effective version of this fact for binary forms of degree $n \geq 2$. This, together with Theorem \ref{mainineq}, brings us to the following conclusion.

  \begin{cor}\label{GC2}
Let $m\geq 1$ and $n \geq 3$ be integers. The set of binary forms $F \in \mathbb{Z}[x , y]$ of degree $n$ with nonzero discriminant for which \eqref{Thueineq} has more than $7n$ primitive solutions  is contained in the union of finitely many equivalence classes, a full set of representatives of which can be effectively determined.
\end{cor}

Brindza, Pint\'{e}r, van der Poorten and Waldschmidt \cite{Pin} obtained an upper bound for the number of large solutions of $|F(x , y)| = m$. They proved that this Thue equation has at most $2n^2 (\omega(m) + 1) + 13n$ primitive solutions $(x , y)$ with
$$
\mathbf{H}(x , y) \geq  21 n^2 M^5 m^{\frac{1}{n-2} + \frac{1}{(n-1)^2}},
$$
where $M = M(F)$ denotes the Mahler measure of $F$ and $\mathbf{H}(x , y) = \max \left(|x| , |y|   \right)$.  
 Gy\H{o}ry \cite{Gyo1} established 
a linear upper bound $25 n$ for the number of primitive  solutions  of inequality $|F(x , y)| <m$, with 
$$
\mathbf{H}(x , y) \geq \left(  {2^{n+1}  n^{n/2} M^n m}{|D|^{1/2}}\right)^{\frac{1}{n-2} + \frac{1}{n^2}}.
$$
Here we will prove 
\begin{thm}\label{G}
Let $F(x , y)$ be an irreducible  binary form with integral coefficients, degree $n$ and  discriminant $D \neq 0$. There are at most 
  \[ \left\{
  \begin{array}{l l}
11n &  \, \textrm{if}\, \, n\geq 3 \\
9n & \, \textrm{if}\, \, n\geq 5
\end{array}
\right. \]
 primitive solutions to inequality $|F(x , y)| \leq m$ with 
$$ y > \left(2^{n/(n-2)} n^{\frac{2n -1}{2n-4}} m^{\frac{1}{n-2}} M(F)\right)^{1+\frac{1}{n-1}}.
$$
\end{thm}

In order to establish our upper bounds, we measure the size of possible solutions $(x , y)$ of our inequality (or equation) by the size of $y$. As we noted before, we regard $(x , y)$ and $(-x , -y)$ as one solution. Therefore, we will only count solutions with positive values of $y$. 

\textbf{Definition} We call a pair of integral solution $(x , y) \in \mathbb{Z}^2$ with $y > 0$
\begin{eqnarray*}
\textrm{small} \, \, &\textrm{if}& \, \, 0 < y < M(F)^2, \\
\textrm{medium} \, \, &\textrm{if}& \, \,    M(F)^2 \leq y < M(F)^{1 + (n-1)^2} \, \, \textrm{and}\\
\textrm{large} \, \, &\textrm{if}& \, \,  M(F)^{1 + (n-1)^2}  \leq y.
\end{eqnarray*}

We will use different techniques to give upper bounds for the number of solutions in each of  the categories. 

In Section \ref{EF}, we introduce some notation and recall some useful known results.
In Section \ref{TH}, we prove Theorems  \ref{THUNineq} and \ref{THUNequ} as corollaries to our main results.
In Section \ref{smallsol}, we focus on counting small solutions
(see Lemmata \ref{smallineq} and \ref{smallequ}).
In Section  \ref{mediumsol}, we will treat 
medium solutions (see Lemma \ref{SC1},   \ref{SC2}  and \ref{SC1G}).  
In Section \ref{LC}, we associate a logarithmic curve $\Phi(t):\mathbb{R}\rightarrow \mathbb{R}^n$ to each Thue inequality (and equation). In Section \ref{Geo}, we explore the asymptotic  properties of this curve.
By studying some geometric properties of the curve $\Phi(t)$,  in Section \ref{EGP}, we establish an exponential gap principle for the norms of $\Phi(\frac{x}{y})$, where $(x , y)$ varies over large solutions of our Thue inequality (or equation).  Finally in Section \ref{LFL}, we use the theory of linear forms in logarithms,  to show that our gap principle guarantees the existence of only few large solutions
(see Lemmata \ref{large=} and  \ref{largegeneral}).  The idea of associating a logarithmic curve to Thue equations originates from Okazaki \cite{Oka}, where he gives an upper bound for the number of solutions to cubic Thue equations of the type $F(x , y) =1$.  This idea has  been modified and  used by the author in \cite{AkhT}, to give an upper bound for the number of solutions to Thue equations $F(x , y) = 1$ of any degree $n \geq 3$.  The general strategy used here is similar to that in \cite{AkhT}, up to an appropriate definition of the logarithmic curve,  however, in this manuscript   totally different means has been used, especially in estimating the quantities that are playing important roles in the theory of linear forms in logarithms.  Moreover, we need to deal  with non-archimedean valuations and S-units here.  For instance, some work of Bugeaud and Gy\H{o}ry \cite{BugG} is heavily used in estimating the quantities in our linear form in logs.  We should also mention that Theorem \ref{exg}, an essential part of this manuscript,  provides an exponential  gap principle which is similar to the one established in \cite{AkhT}. But the proof relies on understanding the technical details of the geometry of  the logarithmic curve $\Phi(x , y)$, as opposed to the much simpler proof in \cite{AkhT} which appeals to Dobrowolski-type lower bounds for the height of algebraic integers. 

In general, we do not expect that $\frac{1}{n}$ times the number of integer solutions  of $|F(x , y)| \leq m$ is larger for binary forms of degree $n = 3, 4$ than for forms of degree $\geq 5$.
 The reason that for $n = 3, 4$, we have obtained a
  larger upper bound for this quantity   is purely technical (see Theorems \ref{exg} and \ref{exg3,4}). As a matter of fact, in \cite{AkhCI} much better upper bounds for  cubic Thue inequalities have been established, by using  the method of Thue and Siegel.
 
 In this manuscript, our main focus is to determine the largest  integer $m$ so that  the upper bound for the number of solutions to  $|F(x , y)| \leq m$ is independent of $m$. Even though establishing a small upper bound is of interest, it is  not our primary goal.

\section{Preliminaries and Notation}\label{EF}

\subsection{Solutions of Thue inequalities as rational approximations}

Let $F(x , y)$ be a binary form of degree $n$, $f(X) = F(X, 1)$ and $\alpha_{1}$, \ldots,  $\alpha_{n} \in \mathbb{C}$  the roots of  $f(X)$.

\textbf{Definition}. A solution $(x , y)$ of $|F(x , y)| \leq m$ is called {\emph related} to $\alpha_{i}$ if
$$
\left| x - \alpha_{i} y\right| = \min_{1\leq j\leq n} \left| x - \alpha_{j} y\right|.
 $$

The following lemma is a version of the Lewis-Mahler inequality \cite{LM}, refined by Bombieri and Schmidt \cite{Bom}.
\begin{lemma}\label{3S}
Let $F$ be a binary form of degree $n \geq 3$ with integer coefficients and  nonzero discriminant $D$. For every pair of integers $(x , y)$ with $ y \neq 0$
$$
\min_{\alpha} \left| \alpha - \frac{x}{y} \right| \leq \frac{2^{n-1} n^{n-1/2} \left(M(F)\right)^{n-2} |F(x , y)|}{|D|^{1/2} |y|^n},
$$
where the minimum is taken over the zeros $\alpha$ of $F(z , 1)$.
\end{lemma}
\begin{proof}
This is Lemma 3 of \cite{Ste}. 
\end{proof}

\subsection{Heights of polynomials and algebraic numbers}

There are several different heights used in this manuscript. Of course these are all related by various well known inequalities. Now we explain some of the relations between these heights.

For the polynomial $G(X) = c (X - \beta_{1})\ldots  (X - \beta_{n})$ with $c \neq 0$, the Mahler measure $M(G)$ is defined by
$$
M(G) = |c| \prod_{i =1}^{n} \max(1 , \left| \beta_{i}\right|).
$$
Let $G$ be a polynomial of degree $n$ and discriminant $D$. Mahler \cite{Mah} showed
\begin{equation}\label{mahD5}
 M(G) \geq \left(\frac{|D|}{n^n}\right)^{\frac{1}{2n -2}}.
 \end{equation}
The Mahler measure of an algebraic number $\alpha$ is defined as the Mahler measure of the minimal polynomial of $\alpha$ over $\mathbb{Q}$.

For an algebraic number $\alpha$,  the (naive) height of $\alpha$, denoted by $H(\alpha)$, is defined by 
$$
H(\alpha) =  H\left(f(X)\right) =\max  \left( |a_{n}|, |a_{n-1}|, \ldots , |a_{0}|\right),
$$
where $f(x) = a_{n}X^{n} + \ldots + a_{1}X + a_{0}$ is the minimal polynomial of $\alpha$ over $\mathbb{Z}$.
We have 
\begin{equation}\label{Lan5}
  {n \choose \lfloor n/2\rfloor}^{-1} H(\alpha) \leq M(\alpha) \leq (n+1)^{1/2} H(\alpha).
\end{equation}
A proof of this fact can be found in \cite{Mahext}.

 \begin{lemma}\label{mahl5}(Mahler \cite{Mah})
 If $a$ and $b$ are distinct zeros of a polynomial $P(X)$ of degree $n$, then we have
\begin{equation*}
  | a - b | \geq \sqrt{3} (n+1)^{-n} M(P)^{-n+1},
  \end{equation*}
 where $M(P)$ is the Mahler measure of $P$.
  \end{lemma}
   Lemma \ref{mahl5} implies a lower bound for the imaginary parts of non-real roots of a polynomial.
   \begin{cor}\label{NCor}
  Let $P(X)$ be a polynomial  of degree $n$. Assume that $P$ has a non-real root $c$. Then the absolute value of the imaginary part of $c$ is greater than or equal to  $$ \frac{\sqrt{3}}{2} (n+1)^{-n} M(P)^{-n+1}.
  $$
  \end{cor}
  \begin{proof}
  In Lemma \ref{mahl5}, take $a = c $ and $b$ equal to the complex conjugate of $c$. 
  \end{proof}

The following is Lemma 4.5 of \cite{AkhT}:
  \begin{lemma}\label{esug}
 Let $f(X) = a_{n}X^{n} + \ldots + a_{1}X + a_{0}$ be an irreducible  polynomial of degree $n$, and with integral coefficients.  Suppose that  $\alpha$ is a root of $f(X) = 0$. For $f'(X)$, the derivative of $f$, we have
$$
2^{-(n-1)^2} \frac{\left|D\right|}{M(f)^{2n -2}}\leq    |f'(\alpha)| \leq \frac{n(n+1)}{2} H(f) \left(\max (1 , |\alpha|)\right)^{n-1},
$$
where $D$ is the discriminant, $M(f)$ is the Mahler measure and $H(f)$ is the naive height of $f$.
 \end{lemma}
 
  Let $\mathbb{K}$ be an algebraic number field of degree $d$ and discriminant $D_{\mathbb{K}}$. 
 Denote by $\bold{M}_{\mathbb{K}}$ the set of places on $\mathbb{K}$. In every place $v$ we choose a valuation $| . |_{v}$ in the following way. If $v$ is infinite and corresponds to an embedding $\sigma: \mathbb{K} \rightarrow \mathbb{C}$ then for every $\rho \in \mathbb{K}$,
  $$
|\rho|_{v} = \left|\sigma(\rho)\right|^{d_{v}/d}, 
$$
where $d_{v}=1$ if $\sigma(\mathbb{K})$ is contained in $\mathbb{R}$, and $d_{v} = 2$ otherwise.
If $v$ is a finite place corresponding to the prime ideal $\mathfrak{p}$ in $\mathbb{K}$ then $|0|_{v} = 0$ and for every nonzero  $\rho \in \mathbb{K}$
$$
|\rho|_{v} = \textrm{Norm}  ( \mathfrak{p})^{-k/d},
$$ 
with $k = \textrm{ord}_{\mathfrak{p}}(\rho)$.
      Then for every $  \rho \in  \mathbb{Q}(\alpha)^{*}$,       we have the \emph{product formula}:
 $$
  \prod_{v \in \bold{M}_{\mathbb{K}}} |\rho|_{v} = 1.
  $$
  Note that $|\rho|_{v} \neq 1$ for only finitely many $v$. 
 We define the \emph{absolute logarithmic height} of an algebraic number $\rho$ as
 \begin{equation}\label{defoflogho1}
 h(\rho) = \frac{1}{2} \sum _{v \in \bold{M}_{\mathbb{K}}} \left|\log |\rho|_{v}\right| . 
 \end{equation}
 This height is called absolute because it is independent of the chosen field that contains  $\rho$. By the product formula
 \begin{equation}\label{defoflogho}
 h(\rho)=  \sum _{v \in \bold{M}_{\mathbb{K}}} \log\max\left( 1,  |\rho|_{v}\right) . 
\end{equation} 
If $\alpha$ is an algebraic number with minimal polynomial $a_{0} (X - \alpha_{1})\ldots (X - \alpha_{n})$ over $\mathbb{Z}$, then it is known that
 \begin{equation}\label{defofloghpoly}
 h(\alpha_{1})= \frac{1}{n} \left( \log |a_{0}| + \sum _{i=1}^{n} \log\max\left( 1,  |\alpha_{i}|\right)\right) . 
\end{equation} 
Note that by \eqref{defoflogho} and \eqref{defofloghpoly}, the absolute logarithmic height of an algebraic number $\alpha_{1}$ and the Mahler measure of the minimal polynomial $P(X)$ of the algebraic number  are related by a simple identity
$$
 h(\alpha_{1}) = \frac{1}{n} M(P).
 $$

 \subsection{Baker's theory of linear forms in logarithms} 

  Suppose that $\mathbb{K}$ is an algebraic number field of degree $d$ over $\mathbb{Q}$ embedded in $\mathbb{C}$. If $\mathbb{K} \subset \mathbb{R}$, we put $\chi = 1$, and otherwise $\chi = 2$. Let  $\gamma_{1} , \ldots, \gamma_{N} \in \mathbb{K}^{*}$,  with absolute logarithmic heights $h(\gamma_{j})$, $1\leq j \leq N$. Let $\log \gamma_{1}$ , $\ldots$ , $\log \gamma_{N}$ be arbitrary fixed non-zero values of the logarithms. Suppose that 
$$
A_{j} \geq \max \{dh(\gamma_{j}) , |\log \gamma_{j}| \}, \  \   1 \leq j \leq N.
$$
Now consider the linear form 
$$
\mathfrak{L} = b_{1}\log\gamma_{1} + \ldots + b_{N}\log\gamma_{N},
$$
with $b_{1}, \ldots , b_{N} \in \mathbb{Z}$ and with the parameter 
$$B = \max \{1 , \max\{b_{j}A_{j}/A_{N}: \  1\leq j  \leq N\}\}.$$ 
Put
$$
\Omega = A_{1} \ldots A_{N},
$$
$$
C(N) = C(N , \chi) = \frac{16}{N!\chi}e^{N}(2N + 1 + 2 \chi)(N + 2) (4N + 4)^{N + 1}\left(\frac{1}{2}eN\right)^{\chi}  ,
$$
$$
C_{0} = \log (e^{4.4N + 7}N^{5.5}d^{2}\log (eN)),
$$
and
$$
W_{0} = \log(1.5eBd\log(ed)).
$$
 We will use the following result of Matveev  \cite{Mat2} to obtain a lower bound for a linear form in logarithms. There are some other versions of the following statement that can be used here (see  \cite{BW} and \cite{Wal}, for example).
 \begin{prop}[Matveev \cite{Mat2}]\label{mat}
If $\log\gamma_{1} , \ldots , \log\gamma_{N}$ are linearly independent over $\mathbb{Z}$ and
$b_{1} \neq 0$, then 
$$
\log|\mathfrak{L}| > -C(N) C_{0} W_{0}d^{2}\Omega.
$$
\end{prop}

\section{ $GL_{2}(\mathbb{Z})$ actions and the proof of Theorems  \ref{THUNineq} and \ref{THUNequ}}\label{TH}

 In this section, we will see that  when the binary  $F$ is transformed  by the action of an element of $GL_{2}(\mathbb{Z})$, the problem of counting solutions remains unchanged, while the Diophantine approximation properties of $F$ can change very drastically.
Let 
$$
A = \left( \begin{array}{cc}
a & b \\
c & d \end{array} \right)
$$
and define the binary form $F_{A}$ by
$$
F_{A}(x , y) = F(ax + by \ ,\  cx + dy).$$

We say that two binary forms $F$ and $G$ are equivalent if $G = \pm F_{A}$ for some $A \in GL_{2}(\mathbb{Z})$

Let $F$ be a binary form that factors in $\mathbb{C}$ as
$$
\prod_{i=1}^{n} (\alpha_{i} x - \beta_{i}y).
$$
The discriminant $D(F)$ of $F$ is given by
$$
D(F) = \prod_{i<j} (\alpha_{i} \beta_{j} - \alpha_{j}\beta_{i})^2.
$$
Observe that for any $2 \times 2$ matrix $A$ with integer entries
\begin{equation}\label{St6}
D(F_{A}) = (\textrm{det} A)^{n (n-1)} D(F).
\end{equation}

Let  $A \in GL_{2}(\mathbb{Z})$. Notice that 
$$A \left( \begin{array}{c} 
x\\
y
\end{array} \right)= \left( \begin{array}{c} 
ax+by\\
cx+dy
\end{array} \right)
$$
and $F_{A^{-1}} ( ax+by , cx+dy ) =  \pm F(x , y)$.
Also, when $\det A = \pm 1$ we have $\gcd(ax+by , cx+dy) =1 $ if and only if $\gcd(x , y) =1$. Therefore, the number of solutions (and the number of primitive solutions) to Thue equations and inequalities does not change if we replace the binary form with an equivalent form.  Moreover the discriminants of two equivalent forms are equal.

Let $p$ be a prime number and put 
$$
A_{0} = \left( \begin{array}{cc}
p & 0 \\
0 & 1 \end{array} \right)\, , \qquad \textrm{and}\qquad
A_{j} = \left( \begin{array}{cc}
0 & -1 \\
p & j \end{array} \right),
$$
for $j = 1, \ldots, p$. Then we have
$$
\mathbb{Z}^2 = \cup_{j=0}^{p} A_{j} \mathbb{Z}^2.
$$
 Therefore, the number of solutions of $|F(x , y)| \leq m$ is at most $N_{F_{0}} + N_{F_{1}} + \ldots + N_{F_{p}}$, where
$$
F_{j} (x , y) = F_{A_{j}}(x , y),
$$
and $N_{F_j}$ is the number of solutions to $|F_{j} (x , y)|\leq m$.
Note that by (\ref{St6}), 
$$
\left| D(F_{A_{j}})\right| \geq p^{n(n-1)} |D(F)|.
$$
 This means if $N$ is an upper bound for the number of solutions to $|G(x , y)| \leq m$, where  $G$ ranges over binary forms of degree $n$ , with $|D(G)| \geq p^{n(n-1)}$, then  $(p + 1)N$ will be an upper bound for the number of solutions to $|F(x , y)| \leq m$ when $F$ has a nonzero discriminant. 
 This argument, together with Theorems \ref{mainineq}  and  \ref{maineq},  leads to the statements of Theorems \ref{THUNineq} and \ref{THUNequ}, by taking $\epsilon = \frac{1}{8n}$ and $\epsilon = \frac{1}{4	n}$, repectively.

Assume  that $|F(x , y)| \leq m$ has a primitive solution $(x_{0} , y_{0})$. Then there is a matrix $A$ in $GL_{2}(\mathbb{Z})$, with $\det A = \pm 1$, for which 
$A^{-1} (x _{0} , y_{0})$ is $(1 , 0)$. Therefore, $(1 , 0)$ is a solution to 
$$|F_{A}( x , y)| \leq m.$$
 We conclude that the leading coefficient of   $F_{A}$ is an integer that does not exceed $m$ in absolute value. From now on we will assume that
 $$
 |a_{0}| \leq m
 $$
    in
 $
 F(x , y)  =  a_{0} x^n + a_{1} x^{n-1}y + \ldots + a_{n}y^n . 
$ 
 Similarly, provided that the equation $|F(x , y)| = m$ has at least one solution, we can (and will) assume that the leading coefficient of $F$ is $\pm m$.

 \section{Small Solutions}\label{smallsol}

\subsection{Small Solutions of Thue Inequalities}

Fix a positive real number $Y_{0}$. Following \cite{EG16}, \cite{Ste} and \cite{Bom}, we will estimate the number of  primitive solutions $(x , y)$ to $|F(x , y)| \leq m$,  for which $0< y \leq Y_{0}$.    Let
\begin{equation}\label{mineq}
0 <  m \leq \frac{|D| ^{\frac{1}{4(n-1)} - \epsilon }}{\left(\frac{7}{2}\right)^{n/2} \, n^{\frac{n}{4(n-1)} }}.  
 \end{equation}
We have assumed that $a_{0}$, the leading coefficient of $F(x , y)$, satisfies 
$$
1 \leq |a_{0}| \leq m.
$$
 We will also assume that $F$  has the smallest Mahler measure among all equivalent  forms that have their leading coefficient equal to $a_{0}$.     For the binary form 
$$
F(x , y) = a_{0} (x - \alpha_{1}y)\ldots  (x - \alpha_{n}y)
$$
put 
$$L_{i}(x , y) = x - \alpha_{i}y
$$
 for $i =1, \ldots, n$. Then
\begin{lemma}\label{S56}
Suppose $(x , y)$ is a primitive solution of $|F(x , y)| \leq m$. We have
$$
\frac{1}{L_{i}(x , y)} - \frac{1}{L_{j}(x , y)} = (\beta_{j} - \beta_{i}) y,
$$
where $\beta_{1}$,\ldots, $\beta_{n}$ are such that the form
$$
J(u , w) = (u - \beta_{1}w)\ldots (u - \beta_{n}w)
$$
is equivalent to $F$.
\end{lemma}
\begin{proof}
 This is Lemma  5 of \cite{EG16}, Lemma 
 4 of \cite{Ste} and Lemma 3 of \cite{Bom}, by taking $(x_{0}, y_{0}) = (1 , 0)$.
\end{proof}

Let  $(x , y) \neq (1 , 0)$  be a solution of $|F(x , y)| \leq m$. We have
 $$\prod_{i=1}^{n}\left| L_{i} (x , y)\right| = \frac{|F(x , y)|}{|a_{0}|} \geq \frac{1}{m}.$$
Fix $j = j(x , y)$ such that
$$
\left|L_{j}(x , y) \right| \geq m^{-\frac{1}{n}}.
$$
Fix distinct indices  $i, j \in \{1, \ldots, n\}$. Then by Lemma \ref{S56},
\begin{equation}\label{S57}
\frac{1}{\left|L_{i}(x , y) \right|} \geq |\beta_{j} - \beta_{i}| |y| - m^{\frac{1}{n}}.
\end{equation}
For the complex conjugate  $\bar{\beta_{j}}$  of $\beta_{j}$, where $j = j(x , y)$,  we also have
$$  
\frac{1}{\left|L_{i}(x , y) \right|} \geq |\bar{\beta_{j}} - \beta_{i}| |y| - m^{\frac{1}{n}}.
$$
Hence
$$
\frac{1}{\left|L_{i}(x , y) \right|} \geq |\textrm{Re}(\beta_{j}) - \beta_{i}| |y| - m^{\frac{1}{n}},
$$
where $\textrm{Re}(\beta_{j})$ is the real part of $\beta_{j}$.
We now choose an integer $\bold{m} = \bold{m}(x , y)$, with $|\textrm{Re}(\beta_{j}) -\bold{m}|\leq 1/2$, and we obtain 
\begin{equation}\label{S58}
\frac{1}{\left|L_{i}(x , y) \right|} \geq \left(|\bold{m}- \beta_{i}| -\frac{1}{2}\right) |y| - m^{\frac{1}{n}},
\end{equation}
for $i = 1,\ldots , n$.

For $1\leq i \leq n$, let $\frak{X}_{i}$ be the set of solutions to $|F(x , y)| \leq m$ with $1\leq y \leq Y_{0}$ and $\left|L_{i}(x , y) \right| \leq \frac{1}{2y}$. Notice that if $\alpha_{i}$ and $\alpha_{j}$ are complex conjugates then $\frak{X}_{i} = \frak{X}_{j}$. So if $F(X , 1) = 0$ has $2q$ non-real roots then we may  consider $n - q$ sets $\frak{X}_{i}$.

\begin{lemma}\label{Sl5}
Suppose $(x_{1} , y_{1})$ and $(x_{2} , y_{2})$ are two distinct solutions in $\frak{X}_{i}$ with $y_{1} \leq y_{2}$. Then
$$
\frac{y_{2}}{y_{1}} \geq \frac{1}{\frac{5}{2} + m^{1/n}} \max(1 , |\beta_{i}(x_{1} , y_{1}) - \bold{m}(x_{1} , y_{1})|).
$$
\end{lemma}
\begin{proof}
This is Lemma 6 of \cite{EG16}.
\end{proof}

\begin{lemma}\label{Sl6}
Suppose $(x , y) \in \mathbb{Z}^2$,  with $y > 0$,  satisfies $|F(x , y)| \leq m$  and 
$\left|L_{i}(x , y) \right|> \frac{1}{2y}$. Then 
$$
|\bold{m}(x , y) - \beta_{i}(x , y)| \leq \frac{5}{2} + m^{1/n}.
$$
\end{lemma}
\begin{proof}
This is Lemma 7 of \cite{EG16}.
\end{proof}

By Lemma \ref{S56}, the form 
$$
J(u , w) = a_{0} (u - \beta_{1}w)\ldots (u - \beta_{n}w)
$$
is equivalent to $F(x , y)$ and therefore the form
$$
\hat{J}(u , w) =  a_{0} (u - (\beta_{1}-\bold{m}) w)\ldots (u - (\beta_{n}- \bold{m})w)
$$
is also equivalent to $F(x , y)$, where $\bold{m}  = \bold{m}(x , y)$. Therefore, since we assumed that $F$ has the smallest Mahler measure among its equivalent forms, we get
\begin{equation}\label{Spre60}
\prod_{i=1}^{n} \max(1, |\beta_{1}(x , y)-\bold{m}(x , y)|) \geq \frac{M(F)}{|a_{0}|}. 
\end{equation}
For each set $\frak{X}_{i}$, ($i = 1, \ldots, n-q$),  that is not empty,  let  $(x^{(i)} , y^{(i)})$ be the element with the largest value of $y$.
Let $\frak{X}$ be the set of solutions of $|F(x , y)| \leq m$, with $1 \leq y \leq Y_{0}$ minus the elements $(x^{(1)} , y^{(1)})$, \ldots, $(x^{(n-q)} , y^{(n-q)})$.
Suppose that for some integer $i$, the set   $\frak{X}_{i}$ is non-empty. Index the elements of $\frak{X}_{i}$ as 
$$(x_{1}^{(i)}, y_{1}^{(i)}), \ldots, (x_{v}^{(i)}, y_{v}^{(i)}),$$
 so that $y_{1}^{(i)} \leq \ldots \leq y_{v}^{(i)}$ (note that $(x_{v}^{(i)}, y_{v}^{(i)}) = (x^{(i)} , y^{(i)})$). By Lemma \ref{Sl5},
\begin{equation*}
\frac{1}{\frac{5}{2} + m^{1/n}} \max\left(1, \left|\beta_{i}(x_{k}^{(i)}, y_{k}^{(i)}) - \bold{m}(x_{k}^{(i)}, y_{k}^{(i)})\right|\right) \leq \frac{y_{k+1}^{(i)}}{y_{k}^{(i)}},
\end{equation*}
for $k = 1 \ldots, v-1$. Hence
$$
\prod_{(x , y) \in  \frak{X} \bigcap \frak{X_{i}} }\frac{1}{\frac{5}{2} + m^{1/n}}  \max\left(1, \left|\beta_{i}(x, y) - \bold{m}(x, y) \right|\right) \leq Y_{0}.
$$ 
For $(x , y)$ in $\frak{X}$ but not in  $\frak{X_{i}}$ we have,  by Lemma \ref{Sl6},
$$
\frac{1}{\frac{5}{2} + m^{1/n}}  \max\left(1, \left|\beta_{i}(x, y) - \bold{m}(x, y) \right|\right) \leq 1.
$$
Thus
\begin{equation*}\label{S59}
\prod_{(x , y) \in  \frak{X}}\frac{1}{\frac{5}{2} + m^{1/n}}   \max\left(1, \left|\beta_{i}(x, y) - \bold{m}(x, y) \right|\right) \leq Y_{0}.
\end{equation*}
Let $|\frak{X}|$ be the cardinality of $\frak{X}$. Comparing the above inequality with (\ref{Spre60}), we obtain 
\begin{equation}\label{S60}
\left( \left(\frac{1}{\frac{5}{2} + m^{1/n}} \right)^{n} \frac{M(F)}{m}\right)^{ |\frak{X}|} \leq Y_{0}^{n-q},
 \end{equation}
for we have at most $n-q$ different sets $\frak{X_{i}}$ .

Let $\theta = 4(n-1) \epsilon$. Then, in view of \eqref{mahD5}, and by \eqref{mineq},
$$
m ^2< \frac{|D| ^{\frac{1}{2(n-1)} (1 -\theta)}}{\left(\frac{7}{2}\right)^n \, n^{\frac{n}{2(n-1)} (1 -\theta)}} \leq \frac{M(F) ^{ 1 -\theta}}{\left(\frac{7}{2}\right)^n}
$$
and therefore,
$$
M(F)^{\theta} \leq \frac{M(F)}{m^2 \left(\frac{7}{2}\right)^n} \leq  \frac{M(F)}{m^2 \left(\frac{5}{2} m^{-1/n} + 1\right)^n}
$$
From this and by (\ref{S60}),
$$
|\frak{X}| \leq \frac{(n-q) \log Y_{0}}{\theta \log M(F)} =\frac{(n-q) \log Y_{0}}{ 4(n-1) \epsilon \log M(F)}.
$$
Thus, when $Y_{0} = M(F)^2$,  we have $|\frak{X}|< \frac{(n-q) }{ 2(n-1) \epsilon }$.  Therefore, we have the following.
\begin{lemma}\label{smallineq}
 The inequality $|F(x , y)| \leq m$ has at most $(n-q) \left(1 + \frac{1 }{ 2(n-1) \epsilon }\right)$ primitive solutions $(x , y)$ with  $0 < y \leq M(F)^2$, provided that  
 $$
0<  m \leq \frac{|D| ^{\frac{1}{4(n-1)} - \epsilon }}{\left(\frac{7}{2}\right)^{n/2} \, n^{\frac{n}{4(n-1)} }}.  $$
\end{lemma}

\subsection{Small Solutions of Thue Equations}
  We estimate the number of  primitive solutions $(x , y)$ to $|F(x , y)| = m$  for which $0< y \leq Y_{0}$.  We keep our assumption
 \begin{equation}\label{mequ}
0 < m \leq \frac{|D| ^{\frac{1}{2(n-1)} - \epsilon }}{\left(\frac{7}{2}\right)^n \, n^{\frac{n}{2(n-1)} }}.  
 \end{equation}  
 We have assumed that $|a_{0}|  = m$. We may also assume that $F$  has the smallest Mahler measure among all equivalent  forms that have their leading coefficient equal to $\pm m$.     Therefore, our equation looks like
$$
F(x , y) = m (x - \alpha_{1}y)\ldots  (x - \alpha_{n}y) = m,
$$
which means
$$
(x - \alpha_{1}y)\ldots  (x - \alpha_{n}y) =  1.$$
Repeating the above argument, the inequality \eqref{S57} will become
\begin{equation}\label{S57*}
\frac{1}{\left|L_{i}(x , y) \right|} \geq |\beta_{j} - \beta_{i}| |y| - 1.
\end{equation}
Consequently, in this case, we can replace  \eqref{S60} by
 \begin{equation}\label{S60*}
\left( \left(\frac{2}{7 } \right)^{n} \frac{M(F)}{m}\right)^{ |\frak{X}|} \leq Y_{0}^{n-q}.
 \end{equation}

Let $\theta' = 2(n-1) \epsilon$. Then, by \eqref{mahD5} and \eqref{mequ},
$$
m < \frac{|D| ^{\frac{1}{2(n-1)} (1 -\theta')}}{\left(\frac{7}{2}\right)^n \, n^{\frac{n}{2(n-1)} (1 -\theta')}} \leq \frac{M(F) ^{ 1 -\theta'}}{\left(\frac{7}{2}\right)^n}
$$
and therefore,
$$
M(F)^{\theta'} \leq \frac{M(F)}{m \left(\frac{7}{2}\right)^n}.$$
From here and by (\ref{S60*}),
$$
|\frak{X}| \leq \frac{(n-q) \log Y_{0}}{\theta' \log M(F)} =\frac{(n-q) \log Y_{0}}{ 2(n-1) \epsilon \log M(F)}.
$$
Thus, when $Y_{0} = M(F)^2$,  we have $|\frak{X}|< \frac{(n-q) }{ 2(n-1) \epsilon }$. So we have  proven
\begin{lemma}\label{smallequ}
 The equation $|F(x , y)| = m$ has at most $(n-q) \left(1 + \frac{1 }{ (n-1) \epsilon }\right)$ primitive solutions $(x , y)$ with  $0 < y \leq M(F)^2$, provided that  
 $$
 0 < m \leq \frac{|D| ^{\frac{1}{2(n-1)} - \epsilon }}{\left(\frac{7}{2}\right)^n \, n^{\frac{n}{2(n-1)} }}.  $$
\end{lemma}

\section{Medium size solutions}\label{mediumsol}

In order to count the number of primitive solutions $(x , y)$, with $M(F)^2 < y < M(F)^{1+ (n-1)^2}$, we will use the Lewis-Mahler inequality (see Lemma \ref{3S}). Also the inequality 
 \begin{equation}\label{Dn}
 \frac{3 + 2\log|D|}{\log 3} \geq n
 \end{equation} (see \cite{Gyo23} for a proof)
 will be used in most of our proofs in this section.
 
 \begin{lemma}\label{SC1}
Let $F(x , y)$ be a binary form with integral coefficients, degree $n$ and  discriminant $D \neq 0$. Suppose $m$ is an integer satisfying 
$$
m < \frac{|D| ^{\frac{1}{2(n-1)} }}{\left(\frac{7}{2}\right)^n \, n^{\frac{n}{2(n-1)}}} $$
and $\alpha_{i}$ is a real root of $F(X , 1) = 0$. 
Then there are at most $2$ primitive solutions to the  inequality $|F(x , y)| \leq m$, with $M(F)^2 < y < M(F)^{1+ (n-1)^2}$, that are related to $\alpha_{i}$.
\end{lemma}
\begin{proof}
Assume that $(x_{1} , y_{1})$, $(x_{2}, y_{2})$ and $(x_{3}, y_{3})$ are three distinct solutions  to $|F(x , y)| \leq m$ and all related to $\alpha_{i}$, with $ y_{3} > y_{2} > y_{1} > M(F)^2$.  Let $j = 1, 2$. By Lemma \ref{3S}, we have
$$
\left| \frac{x_{j+1}}{y_{j+1}} - \frac{x_{j}}{y_{j}} \right| \leq \frac{2^{n} n^{n-1/2} m \left(M(F)  \right)^{n-2}}{|D|^{1/2} |y_{j}|^n}.
$$
Since $(x_{1} , y_{1})$, $(x_{2}, y_{2})$ and $(x_{3}, y_{3})$ are  distinct solutions, we have  $$|x_{j+1}y_{j} - x_{j}y_{j+1}|\geq 1.$$ From our assumption $m < \frac{|D| ^{\frac{1}{2(n-1)} }}{\left(\frac{7}{2}\right)^n \, n^{\frac{n}{2(n-1)}}}$ and by \eqref{Dn}, we have $D^{1/2} > m 2^n n^{n-\frac{1}{2}}$.   Therefore,
$$
\left|\frac{1}{y_{j}y_{j+1}} \right|\leq \left| \frac{x_{j+1}}{y_{j+1}} - \frac{x_{j}}{y_{j}} \right| \leq \frac{M(F)^{n-2} }{ |y_{j}|^n}.
$$
Thus
\begin{equation}\label{S65}
\frac{y_{j}^{n-1}}{M(F)^{n -2}} \leq y_{j+1},
\end{equation}
for $j = 1, 2$. 
Following Stewart \cite{Ste}, we define $\delta_{j}$, for $j = 1, 2, 3$, by
$$
y_{j} = M(F)^{1+\delta_{j}}.
$$
By (\ref{mahD5}), $M(F) > 1$ and so (\ref{S65}) implies that
$$
(n-1) \delta_{j} \leq \delta_{j+1}.
$$ 
This implies  that 
$$
y_{3} \geq M(F)^{1+ (n-1)^2}.
$$
In other words, for each real root $\alpha_{i}$, there are at most $2$ solutions $(x , y)$, with $M(F)^2 < y < M(F)^{1+ (n-1)^2}$, that are related to $\alpha_{i}$. 
\end{proof}

\begin{lemma}\label{Grp}
For a binary form $F(x , y)$ with integer coefficients and degree $n$, let $\alpha$ be a non-real root of $F(x , 1) = 0$. Let $m$ be a positive integer. If a pair of integers $(x , y)$ satisfies
$|F(x , y)| \leq m$ and is related to $\alpha$ then
\begin{equation}\label{AG}
|y| \leq \frac{(n+1)\, m^{1/n}\,   2^{\frac{(n-1)^2}{n}}}{\left(\sqrt{3} \left|D\right|\right)^{1/n}} M(F)^{3-3/n}.
\end{equation}
\end{lemma}
\begin{proof}
 The proof is similar to the proof of Proposition 5.1 of \cite{AkhT}.
\end{proof}

\begin{lemma}\label{SC2}
Let $F(x , y)$ be a binary form with integral coefficients, degree $n$ and  discriminant $D \neq 0$. Suppose $m$ is an integer satisfying 
$$
0 < m < \frac{|D| ^{\frac{1}{2(n-1)} }}{\left(\frac{7}{2}\right)^n \, n^{\frac{n}{2(n-1)} }} $$
and $\alpha_{i}$ is a non-real root of $F(X , 1) = 0$. 
Then  there exists  at most $1$ primitive  solution to the inequality $|F(x , y)| \leq m$, with $M(F)^2 < y < M(F)^{1+ (n-1)^2}$, that is  related to $\alpha_{i}$.
\end{lemma}
\begin{proof}
Assume that $(x_{1} , y_{1})$ and $(x_{2}, y_{2})$  are two distinct solutions  to inequality $|F(x , y)| \leq m$ and both related to $\alpha_{i}$, a non-real root of $F(z , 1) = 0$, with $ y_{2} > y_{1} > M(F)^2$.  In the proof of Lemma \ref{SC1}, we have, similarly to (\ref{S65}),
$$
\frac{y_{1}^{n-1}}{M(F)^{n -2}} \leq y_{2}.
$$
Since $y_{1} > M(F)^2$, we conclude that 
$$
y_{2} \geq M(F)^{n}.
$$
This contradicts (\ref{AG}), since $M(F)$ is large. Therefore, for each non-real root $\alpha_{i}$, there is at most $1$ solution in integers $x$ and $y$, with $M(F)^2 < y < M(F)^{1+ (n-1)^2}$, that is related to 
$\alpha_{i}$.
\end{proof}

  In a similar manner, we show the following statement for the  inequality $|F(x , y)| \leq m$, 
  where we do not assume any restriction for $m$ in terms of the discriminant of $F$.

\begin{lemma}\label{SC1G}
Let $F(x , y)$ be a binary form with integral coefficients, degree $n$ and  discriminant $D \neq 0$. Suppose that $\alpha_{i}$ is a real root of $F(X, 1) = 0$. 
Then   there are at most $3$ primitive solutions to the inequality $|F(x , y)| \leq m$, with 
$$\left(2^{n/(n-2)} n^{\frac{2n -1}{2n-4}} m^{\frac{1}{n-2}} M(F)\right)^{1+\frac{1}{n-1}} < y < \left(2^{n/(n-2)} n^{\frac{2n -1}{2n-4}} m^{\frac{1}{n-2}} M(F)\right)^{1+(n-1)^2},$$
that are related to $\alpha_{i}$.
\end{lemma}
\begin{proof}
Assume that $(x_{j} , y_{j})$ and  $(x_{j+1}, y_{j+1})$  are two distinct solutions  to $|F(x , y)| \leq m$ and all related to $\alpha_{i}$ with $ y_{j+1} \geq y_{j} >2^{n/(n-2)} n^{\frac{2n -1}{2n-4}} m^{\frac{1}{n-2}} M(F)$. Similarly to the proof of Lemma \ref{SC1}, we get
$$
\left| \frac{x_{j+1}}{y_{j+1}} - \frac{x_{j}}{y_{j}} \right| \leq \frac{2^{n} n^{n-1/2} m \left(M(F)  \right)^{n-2}}{|D|^{1/2} |y_{j}|^n}.
$$
In view of \eqref{Dn}, we conclude that 
$$
\left|\frac{1}{y_{j}y_{j+1}} \right|\leq \frac{2^{n} n^{n-1/2} m \left(M(F)  \right)^{n-2}}{|y_{j}|^n}.
$$
Thus,
\begin{equation}\label{S65G}
\frac{y_{j}^{n-1}}{\mathfrak{B}} \leq y_{j+1}, 
\end{equation}
where $\mathfrak{B} = 2^{n} n^{n-1/2} m \left(M(F)  \right)^{n-2} $.
We define $\delta_{j}$ as follows:
$$
y_{j} =\left(2^{n/(n-2)} n^{\frac{2n -1}{2n-4}} m^{\frac{1}{n-2}} M(F)\right)^{1+\delta_{j}}.
$$
Inequality  \eqref{S65G} implies that
$$
(n-1) \delta_{j} \leq \delta_{j+1}.
$$ 
 From this we conclude that if there are $4$ solutions $(x_{j} , y_{j})$  with 
 $y_{4} > y_{3} > y_{2} > y_{1} > \left(2^{n/(n-2)} n^{\frac{2n -1}{2n-4}} m^{\frac{1}{n-2}} M(F)\right)^{1+\frac{1}{n}}$, then 
$$
y_{4} \geq  \left(2^{n/(n-2)} n^{\frac{2n -1}{2n-4}} m^{\frac{1}{n-2}} M(F)\right)^{1+(n-1)^2}.
$$
In other words, for each  real root $\alpha_{i}$, there are at most $3$ solutions $(x , y)$ satisfying 
$$\left(2^{n/(n-2)} n^{\frac{2n -1}{2n-4}} m^{\frac{1}{n-2}} M(F)\right)^{1+\frac{1}{n-1}} < y < \left(2^{n/(n-2)} n^{\frac{2n -1}{2n-4}} m^{\frac{1}{n-2}} M(F)\right)^{1+(n-1)^2},$$
that are related  to $\alpha_{i}$
 \end{proof}

 \section{The proof of Theorem \ref{mainred}}

In this short section, we turn our attention to the reducible binary forms. The bounds established in  the previous sections can be easily used here, as no assumption on irreducibility of binary forms has been made yet. 

The following is Lemma 1 of \cite{EG1}, which guarantees no large solution and fewer medium solutions for the inequality, when our binary form is reducible.
\begin{lemma}[Evertse and Gy\H{o}ry]\label{redEG}
Assume that $F$ is reducible. Then every solution of the inequality $|F(x , y)| \leq m$ satisfies
$$
|y| \leq 2^{n^2/2} M(F)^{n}.
$$
\end{lemma}

From Lemma \ref{redEG}, we conclude that the inequality $|F (x , y)|\leq m$ has no ``large" solution, if $F$ is reducible. The proof of Lemma \ref{SC1} shows that there cannot be more than one ``medium'' solution related to each root of $F$ in this case.  Therefore, our inequality $|F(x, y)| < m$ has at most $n$ ``medium'' solutions.
 By Lemma \ref{smallineq}, the inequality has at most $n \left(1 + \frac{1 }{ 2(n-1) \epsilon }\right)$ solutions if
$
0<  m \leq \frac{D^{\frac{1}{2(n-1)} - \epsilon} }  {(3.5)^{n/2} n^{ \frac{n}{2(n-1) } } }.
 $

 \section{The Logarithmic Curve $\Phi$} \label{LC}

In the rest of this paper, we will study the number of large solutions to $|F(x , y)| = m$ and  $|F(x , y)| \leq m$, where $F(x , y) \in \mathbb{Z}[x , y]$ is an irreducible binary form of degree $n \geq 3$ and $m$ is a positive integer. Denote by $D$ the discriminant of $F$ and put $f(X) = F(X , 1)$. The case that $f$ has only non-real roots has been settled in Lemma  \ref{Grp}.
From now on, we  assume that $f$ has at least one real root.
Define, for $k \in \{1, 2, \ldots, n\}$,
\begin{equation}\label{fi}
\phi_{k}(x , y) = \log \left| \frac{D^{\frac{1}{n(n-2)}} \, (x - y\alpha_{k})}{ F(x , y)^{1/n}\left( f'(\alpha_{k})\right)^{\frac{1}{n-2} }}\right|,
\end{equation}

\begin{equation}\label{fgs}
\Phi(x , y) = \left( \phi_{1}(x , y) , \phi_{2}(x , y), \ldots , \phi_{n}(x , y) \right)  \in \mathbb{R}^n.
\end{equation}
  
  In the remainder of Section \ref{LC}, we assume that 
  \begin{equation}\label{Referee1}
  |a_{0}| \leq m,
  \end{equation}
  \begin{equation}\label{Referee2}
  |D| > 2^{(n-1)^2}.
  \end{equation}
  The first assumption simply means that $|F(1 , 0)| \leq m$. The second assumption does not hold for all binary forms, but it is a consequence of our assumption on $m$ in relation to $D$ in our main theorems. We denote by $\| .\|$ the Euclidean norm on $\mathbb{R}^n$.
  \begin{lemma}\label{lem10}
Assume \eqref{Referee1} and \eqref{Referee2}. Then
$$
 \left\| \Phi(1 , 0) \right\| \leq \sqrt{n}  \log \left(|D|^{\frac{1}{n(n-2)}} M(F)^{\frac{2n -2}{n-2}}\right).
 $$
\end{lemma}
\begin{proof}
By Lemma \ref{esug},
$$
|f'(\alpha_{k})| \geq 2^{-(n-1)^2} \frac{\left|D\right|}{M(F)^{2n -2}}.
$$
Since we assumed that   $D >2^{(n-1)^2}$, we get
$$
|f'(\alpha_{k})| \geq  \frac{1}{M(F)^{2n -2}}.
$$
Therefore, for $1 \leq k \leq n$,
\begin{eqnarray*}
\phi_{k} (1, 0) &=& \log \left| \frac{ D^{\frac{1}{n(n-2)}}}{F(1 , 0)^{1/n} \left| f'(\alpha_{k})\right|^{\frac{1}{n-2} }}\right| \\
&\leq&  \log \left(|D|^{\frac{1}{n(n-2)}} M(F)^{\frac{2n -2}{n-2}}\right).
\end{eqnarray*} 
By the definition of $\Phi$  in (\ref{fgs}),
$$
\left\| \Phi(1 , 0) \right\| \leq \sqrt{n}  \log \left(|D|^{\frac{1}{n(n-2)}} M(F)^{\frac{2n -2}{n-2}}\right).
$$
\end{proof}

  \begin{lemma}\label{lem1}
Assume \eqref{Referee1} and \eqref{Referee2}. Suppose that $(x ,y)$ satisfies the equation $\left|F(x , y)\right| \leq m$ and
$$
\left| x - \alpha_{i} y\right| = \min_{1\leq j\leq n} \left| x - \alpha_{j} y\right|.
 $$
Then 
$$
 \left\| \Phi(x , y) \right\| \leq n\sqrt{n} \log \frac{|F(x , y)|^{1/n}}{ \left|  x - \alpha_{i} y\right|} + \sqrt{n} \log \left(  |D|^{\frac{1}{n(n-2)}} M(F)^{\frac{2n -2}{n-2}} \right).
 $$
 \end{lemma}
\begin{proof}
To estimate the size of $\Phi(x , y)$, we write it as the  sum of two vectors $\mathbf{v}_{1}$, $\mathbf{v}_{2} \in \mathbb{R}^n$,
where
$$\mathbf{v}_{1} =  \left(\log \left| \frac{  D^{\frac{1}{n(n-2)}}}{\left| f'(\alpha_{1})\right|^{\frac{1}{n-2} }}\right|, \ldots,  \log \left| \frac{  D^{\frac{1}{n(n-2)}}}{\left| f'(\alpha_{n})\right|^{\frac{1}{n-2} }}\right|   \right) $$
 and 
 $$
\mathbf{v}_{2} = \mathbf{v}_{2}(x , y)=  \left(\log \left| F(x , y))^{-1/n} (x - \alpha_{1} y)\right|, \ldots,  \log \left| F(x , y)^{-1/n} (x - \alpha_{n} y)\right|\right). $$
 Similarly  to  the proof of Lemma \ref{lem10}, we have
$$
 \|\mathbf{v}_{1}\|^2 = \sum_{k=1}^{n} \log^{2} \left| \frac{  D^{\frac{1}{n(n-2)}}}{\left| f'(\alpha_{k})\right|^{\frac{1}{n-2} }}\right| \leq  n  \log^2 \left(  |D|^{\frac{1}{n(n-2)}} M(F)^{\frac{2n -2}{n-2}} \right). $$

Since 
$$|F(x , y)| = |a_{0}|\prod_{1\leq j\leq n} \left| x - \alpha_{j} y\right| \leq m 
$$
 and 
$$
\left| x - \alpha_{i} y\right| = \min_{1\leq j\leq n} \left| x - \alpha_{j} y\right|,
 $$
we have 
$$
\left| F(x , y)^{-1/n} (x - \alpha_{i} y)\right| \leq\left| \left(\frac{a_{0}}{F(x , y)}\right)^{1/n} (x - \alpha_{i} y)\right| \leq 1.
$$
After a permutation of roots $\alpha_{j}$, we may assume that 
$$
\left| (F(x , y)^{-1/n}  ( x - \alpha_{j} y)\right| \leq 1, \qquad  \textrm{for}  \   j=1, \ldots, u, 
$$
 $$
\left|F(x , y)^{-1/n} ( x - \alpha_{j} y)\right| > 1, \qquad \textrm{for} \  j=u+1\ldots, n
$$
where $u \geq 1$.
Since 
  $$
\left| x - \alpha_{i} y\right| = \min_{1\leq j\leq n} \left| x - \alpha_{j} y\right|,
 $$
we  have
  $$
\left|\log \left|F(x , y)^{-1/n}  (x - \alpha_{j} y)\right| \right| \leq \left| \log  \left| F(x , y)^{-1/n}  (x - \alpha_{i} y)\right| \right|
 $$ 
 for $j=1, \ldots, u$.
Also  
\begin{eqnarray*}
& & \prod_{k=u+1}^{n} \left| F(x , y)^{-1/n} (x - \alpha_{k} y)\right| \leq \\
& & |a_{0}|\prod_{k=u+1}^{n} \left|  F(x , y)^{-1/n} (x - \alpha_{k} y)\right| = \frac{1} {\prod_{j} \left|F(x , y)^{-1/n} ( x - \alpha_{s_{j}} y)\right|} .
\end{eqnarray*}
Therefore, for any $k$ with $u+1 \leq k \leq n$, we have
$$
\log \left| F(x , y)^{-1/n} ( x - \alpha_{b_{k}} y) \right| \leq u \log  \frac{1}{\left|  F(x , y)^{-1/n} (x - \alpha_{i} y)\right|}.
 $$
 Therefore,
 \begin{eqnarray*}
 \|\mathbf{v}_{2} \|^2 &\leq&  (n -u)u^2 \left| \log\left|F(x , y)^{-1/n} (x - \alpha_{i} y)\right| \right|^2 +\\ \nonumber
 & +& u \left|  \log\left| (F(x , y)^{-1/n} (x - \alpha_{i} y)\right| \right|^2  \\ \nonumber
  & = &
\left((n -u)u^2 + u\right) \left| \log\left| F(x , y)^{-1/n} (x - \alpha_{i} y)\right| \right| .
 \end{eqnarray*}
 Since
  $$
\left(  (n-u)u^2 + u\right)^{1/2} \leq n \sqrt{n},$$
 for $u = 0, \ldots, n$,  the statement of our lemma follows immediately.
  \end{proof}

\begin{lemma}\label{lem100}
Assume \eqref{Referee2}. Let  $(x , y) \in \mathbb{Z}^2$ be a pair with   $|F(x , y)| \leq m$ and
$$
y \geq   2 \, |a_{0}|^{1/(n-1)} n^3 m^{1/n}  H(F)^{1/(n-2)} M(F) ^{4 + 4\sqrt{n}}.
$$
Then
$$
\left\| \Phi(1 , 0) \right\| < \left\| \Phi(x , y) \right\|.
$$
\end{lemma}
\begin{proof}
Let $\alpha_{1}$, $\ldots$, $\alpha_{n}$ be the roots of $F(X , 1) = 0$. Then 
$$
(\frac{x}{y} - \alpha_{1}) \ldots (\frac{x}{y} - \alpha_{n}) = \frac{\pm F(x , y)}{a_{0} y^n}.
$$
By Lemma \ref{3S}, there must exist a root $\alpha_{j}$ so that 
\begin{eqnarray*}
\left|\frac{x}{y} - \alpha_{j}\right| &\geq& \left| \frac{ F(x , y) D^{1/2} y^n }{a_{0} y^n   2^{n-1} n^{n-1/2} \left(M(F)\right)^{n-2} F(x , y)} \right|^{1/(n-1)}\\
& =& \left| \frac{  D^{1/2} }{a_{0} 2^{n-1} n^{n-1/2} \left(M(F)\right)^{n-2}} \right|^{1/(n-1)}.
\end{eqnarray*}
By Lemmata  \ref{esug} and \ref{3S},
 \begin{eqnarray*}
& &\phi_{j}(x , y) = \log \left| \frac{D^{\frac{1}{n(n-2)}}\,  (x - y\alpha_{j})}{F(x , y)^{1/n}\left( f'(\alpha_{j})\right)^{\frac{1}{n-2} }}\right|\\
&\geq& \log \left| \frac{D^{\frac{1}{n(n-2)}}\,  y \,    D^{1/2(n-1)}  }  { F(x , y)^{1/n} \left(a_{0} 2^{n-1} n^{n-1/2}
 \left(M(F)\right)^{n-2} \right)^{1/(n-1)}  \left(   \frac{n(n+1)}{2} H(F) M(F)^{n-1} \right)^{\frac{1}{n-2}} }\right|\\
& =&   \log |y|  \left| \frac{  D^{\frac{1}{n(n-2) } + 1/2(n-1)} }  { F(x , y)^{1/n} 2 a_{0}^{1/(n-1)}  \left(M(F)\right)^{\frac{n-2}{n-1} + \frac{n-1}{n-2}} H(F)^{1/(n-2)}
 n^{\frac{n-1/2}{n-1}}
  \left(   \frac{n(n+1)}{2}  \right)^{\frac{1}{n-2}} }\right|.
\end{eqnarray*} 
Therefore, we conclude that if
$|y| \geq   2 |a_{0}|^{1/(n-1)} n^3 m^{1/n}  H(F)^{1/(n-2)} M(F) ^{4 + 4\sqrt{n}}$, then 
$|\phi_{j}(x , y)|$  exceeds $\sqrt{n} \log \left( |D|^{\frac{1}{n(n-2)}} M(F)^{\frac{2n -2}{n-2}}\right)$. By Lemma \ref{lem10},
 our proof is complete. 
\end{proof}

\begin{lemma}\label{Dr}
Among all non-zero integer solutions to $|F(x , y)| \leq m$, let $(x_{0}, y_{0})$ be one  for which $\|\Phi(x_{0} , y_{0}\|$ is minimal. Let $(x , y) \neq \pm (x_{0} , y_{0}) \in \mathbb{Z}^2$ be another non-zero solution to $|F(x , y)|\leq m$. Then
  $$
  \left\| \Phi(x , y) \right\| \geq \frac{1}{2} \log \left(\frac{|D|^{\frac{1}{n(n-1)}}}{2 m^{\frac{2}{n}}}\right).
  $$
   \end{lemma}
   \begin{proof}
   Let   $\alpha_{i}$ and $\alpha_{j}$ with $1\leq i, j \leq n$  be two distinct roots of $f(X) = F(X , 1)$. We have
\begin{eqnarray*}
& & \left| e^{\phi_{i}(x_{0} , y_{0}) - \phi_{i}(x , y)} - e^{\phi_{j}(x_{0} , y_{0}) - \phi_{j}(x , y)}\right| \\
& =& \left|\frac{x_{0} - y_{0}\alpha_{i}}{x - y\alpha_{i}} - \frac{x_{0} - y_{0}\alpha_{j}}{x - y\alpha_{j}}\right| \times \left|\frac{F(x , y)}{F(x_{0}, y_{0})} \right|^{1/n}\\
& = & \frac{\left|\alpha_{i} - \alpha_{j}\right| \left|xy_{0} - yx_{0}\right|}{|x - y\alpha_{i}||x - y\alpha_{j}|} \times \left|\frac{F(x , y)}{F(x_{0}, y_{0})} \right|^{1/n}\\
&\geq &\frac{\left|\alpha_{i} - \alpha_{j}\right| }{|x - y\alpha_{i}||x - y\alpha_{j}|} \times \left|\frac{F(x , y)}{F(x_{0}, y_{0})} \right|^{1/n},
\end{eqnarray*}
where $\phi_{i}$'s are defined in \eqref{fi}.
  The last inequality follows from the fact that  $ \left|xy_{0} - yx_{0}\right|$ is a non-zero integer. Since for every $(x , y) \in \mathbb{Z}^2$, we have $|\phi_{i}(x , y)| < \|\Phi (x , y)\|$ and $\left\| \Phi(x_{0} , y_{0}) \right\| < \left\| \Phi(x , y) \right\|$, we  conclude
\begin{eqnarray*}
& &\left( 2e^{2\left\| \Phi(x , y) \right\|}\right)^{\frac{n(n-1)}{2}} \\
& \geq & \prod_{1\leq i < j \leq n}\left| e^{\phi_{i}(x_{0} , y_{0}) - \phi_{i}(x , y)} - e^{\phi_{j}(x_{0} , y_{0}) - \phi_{j}(x , y)}\right|\\
&\geq & \prod_{1 \leq i < j\leq n}\left|\frac{x_{0} - y_{0}\alpha_{i}}{x - y\alpha_{i}} - \frac{x_{0} - y_{0}\alpha_{j}}{x - y\alpha_{j}}\right|  \times \left|\frac{F(x , y)}{F(x_{0}, y_{0})} \right|^{\frac{n-1}{2}}\\
&\geq& \prod_{1 \leq i < j\leq n}\frac{\left|\alpha_{i} - \alpha_{j}\right| }{|x - y\alpha_{i}||x - y\alpha_{j}|}\times \left|\frac{F(x , y)}{F(x_{0}, y_{0})} \right|^{\frac{n-1}{2}}\\
& = &\frac{\sqrt{|D|}}{|F(x , y)|^{n-1}} \times \left|\frac{F(x , y)}{F(x_{0}, y_{0})} \right|^{\frac{n-1}{2}}\\
&\geq & \frac{\sqrt{|D|}}{m^{n-1}}.
\end{eqnarray*}
 Our lemma follows from this.
   \end{proof}

  \begin{lemma}\label{Dr=}
  Among all non-zero integer solutions to $|F(x , y)| = m$, let $(x_{0}, y_{0})$ be one  for which $\|\Phi(x_{0} , y_{0}\|$ is minimal. Let $(x , y) \neq \pm (x_{0} , y_{0}) \in \mathbb{Z}^2$ be another non-zero solution to $|F(x , y)|= m$.
 Then
  $$
  \left\| \Phi(x , y) \right\| \geq \frac{1}{2} \log \left(|D|^{\frac{1}{n(n-1)}}\right).
  $$
   \end{lemma}
   \begin{proof}
   Notice that $ F(x , y) = \pm m (x - \alpha_{1} y) \ldots  (x - \alpha_{n} y) = m$
   and therefore  $$\left|(x - \alpha_{1} y) \ldots  (x - \alpha_{n} y)\right| = 1.$$
     The rest of proof is similar to the proof of Lemma \ref{Dr}.
      \end{proof}

\section{Distance Functions}\label{Geo}

 In Section \ref{LFL}, we will apply Baker's theory of linear forms in logarithms to derive a lower bound for $\left|\log \frac{(\alpha_{n} - \alpha_{i}) (\frac{x}{y} - \alpha_{j})}{(\alpha_{n} -\alpha_{j})( \frac{x}{y}- \alpha_{i} )}\right|$.  First we will establish  an upper bound.
 Set
\begin{equation}\label{bis}
\mathbf{b_{i}} = 
\frac{1}{n}(-1, \ldots, -1, n-1, -1, \ldots , -1).
\end{equation}
Upon noticing that 
$$
\left|\prod_{i=1}^{n} f'(\alpha_{i}) \right| = \left| \frac{D}{a_{0}^{n-2}}\right|,
$$
we may write
$$
\Phi(x , y) = \sum_{i=1}^{n} \log \frac{|\frac{x}{y} - \alpha_{i}|}{\left| f'(\alpha_{i})\right|^{\frac{1}{n-2}}} \bf{b_{i}}.
$$
The new basis $\{ \mathbf{b_{1}}, \ldots, \mathbf{b_{n}}\}$ of $\mathbb{R}^n$ is useful in many ways. Writing 
$\Phi(x , y)$ in this coordinate system helps us to understand this curve as a one variable curve $\Phi(t)$ defined on the real numbers, where $\Phi(x , y) = \Phi(x / y)$ for $x , y \in\mathbb{Z}$. Also, unlike the original definition of $\Phi(x , y)$ with respect to the standard basis of $\mathbb{R}^n$,  the value of $F(x , y)$ does not appear in the coordinates of the vector $\Phi(x , y)$.

It turns out that if  $(x_{1}, y_{1})$ and $(x_{2}, y_{2})$  are two solutions of  $|F(x , y)| \leq m$ that are related to a fixed real root $\alpha_{n}$, with $y_{1}$ and $y_{2}$ large enough, then two points $\Phi(x_{1}, y_{1})$ and $\Phi(x_{2}, y_{2})$ are located near each other. We prove this fact by finding a line $L_{n}$ in $\mathbb{R}^n$ such that  $\Phi(x_{i}, y_{i})$ are all  close to $L_{n}$ . 
Let  $(x , y) \in \mathbb{Z}^2$ be a solution to $|F(x , y)| \leq m$ that is related to $\alpha_{n}$, i.e.,
$$
\left| x - \alpha_{n} y\right| = \min_{1\leq j\leq n} \left| x- \alpha_{j} y\right|.
$$
Then
\begin{equation}\label{E}
\Phi(x , y)  = \sum_{i=1}^{n-1} \log \frac{|\frac{x}{y} - \alpha_{i}|}{\left| f'(\alpha_{i})\right|^{\frac{1}{n-2}}} \mathbf{ c_{i}}
 + E_{n}\mathbf{b_{n}},
\end{equation}
where, for $1\leq i \leq n-1$,
\begin{equation}\label{16*}
{\bf c_{i}} = {\bf b_{i}} 
+ \frac{1}{n-1} {\bf b_{n}} , \quad 
 E_{n} = \log \frac{\left| \frac{x}{y} - \alpha_{n} \right|}{\left|f'(\alpha_{n})\right|^{\frac{1}{n-2}}} - \frac{1}{n-1} \sum_{i=1}^{n-1} \log \frac{|\frac{x}{y} - \alpha_{i}|}{\left| f'(\alpha_{i})\right|^{\frac{1}{n-2}}}.
\end{equation}
One can easily observe that, for   $1\leq i \leq n$,
\begin{equation}\label{ci}
\mathbf{c_{i}} \perp \mathbf{b_{n}}, \  \textrm{and}\  
 \|\mathbf{c_{i}} \| = \frac{\sqrt{n^2 - 3n +2}}{n-1}.
\end{equation}
The above orthogonal coordinate system allows us to see the curve $\Phi (x , y)$ as a one variable curve $\Phi\left(t \right)$ defined on $\mathbb{R}$, where for $x, y \in \mathbb{Z}$, we have $\Phi (x , y) = \Phi \left(\frac{x} { y}\right)$. The next two Lemmata show that $\Phi(t)$ approaches a line $\mathbf{L_{n}}$ as $t$ approaches a fixed root $\alpha_{n}$ of $F(X, 1) = 0$.

\begin{lemma}\label{ATL}
Let
$$
\mathbf{L_{n}} = \left\{ \sum_{i=1}^{n-1} \log \frac{|\alpha_{n} - \alpha_{i}|}{\left| f'(\alpha_{i})\right|^{\frac{1}{n-1}}} \mathbf{c_{i}} + z \mathbf{b_{n}} , \quad z \in \mathbb{R} \right\}.
$$
Let $(x , y) \in \mathbb{Z}^2$ be a  solution to $|F(x , y)|\leq m$ and suppose that
$$
\left| x - \alpha_{n} y\right| = \min_{1\leq j\leq n} \left| x - \alpha_{j} y\right|. $$ 
Then the distance between $\Phi(x, y)$ and the line $\mathbf{L_{n}}$ is equal to
$$
\left\|  \sum_{i=1}^{n-1} \log \frac{|\frac{x}{y}- \alpha_{i}|}{|\alpha_{n} - \alpha_{i}|} \mathbf{c_{i}} \right\|.
$$
Moreover,
\begin{equation}\label{d1}
\left\|  \sum_{i=1}^{n-1} \log \frac{|\frac{x}{y}- \alpha_{i}|}{|\alpha_{n} - \alpha_{i}|} \bf{c_{i}} \right\| <  \frac{2\, M(F)^{n-1} (n+1)^{n}  \sqrt{n (n^2 - 3n +2)}}{\sqrt{3}(n-1)} \, |\frac{x}{y}- \alpha_{n}|.
\end{equation}
\end{lemma}
\begin{proof}
The proof is similar to the proof of Lemma 8.1 of \cite{AkhT}.
\end{proof}

\begin{lemma}\label{gp1}
Let the notation be as in Lemma \ref{ATL} and let $C$ be a positive number.
Suppose that  $(x , y) \in \mathbb{Z}^2$ is a solution to $|F(x , y)| \leq m$  that is related to 
$\alpha_{n}$, with
$$
  | y| > C\frac{2}{\sqrt{3}} (n+1)^{n}  \sqrt{n} \,  m^{n}  |D|^{\frac{1}{n^2 (n-2)}} M(F)^{\frac{2n -2}{n(n-2)} + n-1}. 
  $$ 
 Then the distance between $\Phi(x, y)$ and the line $\mathbf{L_{n}}$ is less than 
$$
\frac{1}{C} \exp\left(\frac{-\left\| \Phi(x , y) \right\|}{n\sqrt{n}}\right).
$$
\end{lemma}
\begin{proof}
 By Lemma \ref{lem1} we have
$$
\left\| \Phi(x , y) \right\| - \sqrt{n}  \log \left( |D|^{\frac{1}{n(n-2)}} M(F)^{\frac{2n -2}{n-2}}\right)
\leq n\sqrt{n} \log \frac{\left| F(x , y) \right|^{1/n} }{  \left| x - \alpha_{n} y\right|} ,
 $$
which implies
$$
  \log \left|\frac{y(\frac{x}{y} - \alpha_{n})}{F(x , y)^{1/n}}\right| \leq - \frac{\left\| \Phi(x , y) \right\|}{n\sqrt{n}} +\frac{1}{n} \log \left( |D|^{\frac{1}{n(n-2)}} M(F)^{\frac{2n -2}{n-2}}\right).
 $$
Therefore, 
\begin{eqnarray} \label{lessthanexp}
& & |\frac{x}{y} - \alpha_{n}| <\\ \nonumber
& &  \exp \left(\frac{-\left\| \Phi(x , y) \right\|}{n\sqrt{n}}\right) \frac{   |D|^{\frac{1}{n^2 (n-2)}} M(F)^{\frac{2n -2}{n(n-2)}}|F(x , y)|^{1/n}}{|y|} .
\end{eqnarray}
By (\ref{d1}) and our assumption on the size of $|y|$, our proof is complete.
\end{proof}

 \begin{lemma}\label{Tnew}
 Let $F(x , y) \in \mathbb{Z}[x , y]$ be a form of degree $n$.  Let $\alpha_{1}$, \ldots, $\alpha_{n}$ be the roots of  $f(X) = F(X , 1)$. Fix a root  $\alpha_{n}$. For any $(x , y) \in \mathbb{Z}^{2}$, with $y \neq 0$, set
  $t = \frac{x}{y}$ and 
 $$
\beta_{i} = \left \{ \begin{array}{ll}
\alpha_{i} & \mbox{if  $i \leq n-1$} \\
\alpha_{i-n+1} & \mbox{ if $i \geq n$.}
\end{array}
\right.
$$ 
Then
$$
\sum_{k=1}^{n-2}\sum_{i=1}^{n-1} \log^2 \left| \frac{(t - \beta_{i}) (\alpha_{n} - \beta_{i+k})}{(\alpha_{n} -\beta_{i}) ( t- \beta_{i+k})}\right|    =  (2n - 2) \left\|  \sum_{i=1}^{n-1} \log \frac{|t- \alpha_{i}|}{|\alpha_{n} - \alpha_{i}|} \bf{c_{i}} \right\| ^2.$$
 \end{lemma}
 \begin{proof}
The proof is similar to the proof of Lemma 8.2 of \cite{AkhT}.
\end{proof}

Let $\alpha_{n}$ be a fixed real root of  $F(X , 1) = 0$. For  any two roots $\alpha_{i}$ and 
$\alpha_{j}$ of $F(X , 1) = 0$ distinct to each other and distinct to $\alpha_{n}$,  and $(x , y) \in \mathbb{Z}^2$, we define 
  \begin{equation}\label{Tdef}
 T_{i,j} (x , y) = T_{i,j}(\frac{x}{y})  =   \log \left|\frac{\alpha_{n} - \alpha_{i}}{\alpha_{n} -\alpha_{j}}\right| + \log \left|\frac{\frac{x}{y} - \alpha_{j}}{\frac{x}{y}- \alpha_{i} }\right|.
 \end{equation}

 \begin{lemma}\label{Tu}
Fix a positive number $C$. Let $(x , y)$ be a pair of integers satisfying  $|F(x , y)| \leq m$ with
$$
  y > C\frac{2}{\sqrt{3}} (n+1)^{n}  \sqrt{n} \,  m^{n}  |D|^{\frac{1}{n^2 (n-2)}} M(F)^{\frac{2n -2}{n(n-2)} + n-1}.
  $$ 
 Then there exists a pair of indices $(i , j)$ with $i \neq j$ for which 
 $$
\left|T_{i, j}(x , y) \right|  < C^{-1}\, \sqrt{ \frac{2}{n-2}} \exp\left(-\frac{\left\| \Phi(x , y) \right\|}{n\sqrt{n}}\right).$$
  \end{lemma}
\begin{proof}
By  Lemma \ref{Tnew}, there must be a pair  $(i , j)$, for which the following holds:
 \begin{eqnarray*}
 & & \log^2 \left|\frac{(t - \alpha_{i}) (\alpha_{n} - \alpha_{j})}{(t - \alpha_{j}) (\alpha_{n} -\alpha_{i})}\right| \\
 & \leq &\frac{1}{(n-1)(n-2)} \sum_{k=1}^{n-2}\sum_{i=1}^{n-1} \log^2 \left| \frac{(t - \beta_{i}) (\alpha_{n} - \beta_{i+k})}{(\alpha_{n} -\beta_{i}) ( t- \beta_{i+k})}\right|     \\ 
  & = & \frac{2(n -1)}{(n-1)(n-2)} \left\|  \sum_{i=1}^{n-1} \log \frac{|t- \alpha_{i}|}{|\alpha_{n} - \alpha_{i}|} \bf{c_{i}} \right\| ^2 .
  \end{eqnarray*}
  Therefore, by Lemmata \ref{ATL} and  \ref{gp1},
  $$
\left|T_{i, j}(x , y) \right|=  \left| \log \left|\frac{(t - \alpha_{i}) (\alpha_{n} - \alpha_{j})}{(t - \alpha_{j}) (\alpha_{n} -\alpha_{i})}\right| \right|  < C^{-1}\, \sqrt{ \frac{2}{n-2}} \exp\left(- \frac{\left\| \Phi(x , y) \right\|}{n\sqrt{n}}\right).
  $$
\end{proof}

\begin{lemma}\label{Tu+pi} Fix a positive number $C$. Let  $\alpha_{n}$ be  a fixed real root of $F(X , 1)$.  
Let $(x , y) \in \mathbb{Z}^2$ be a  solution to $|F(x , y)|\leq m$ and suppose that
$$
\left| x - \alpha_{n} y\right| = \min_{1\leq j\leq n} \left| x - \alpha_{j} y\right|,
 $$ 
$$
  y> C\frac{2}{\sqrt{3}} (n+1)^{n}  \sqrt{n} \,  m^{n}  |D|^{\frac{1}{n^2 (n-2)}} M(F)^{\frac{2n -2}{n(n-2)} + n-1}.
  $$ 
 If $\log \frac{(\alpha_{n} - \alpha_{i}) (\frac{x}{y} - \alpha_{j})}{(\alpha_{n} -\alpha_{j})( \frac{x}{y}- \alpha_{i} )}$ has its principal value then there exists a pair of indices $(i , j)$ with $i \neq j$ for which 
 
$$
 \left|\log \frac{(\alpha_{n} - \alpha_{i}) (\frac{x}{y} - \alpha_{j})}{(\alpha_{n} -\alpha_{j})( \frac{x}{y}- \alpha_{i} )}\right| \leq \sqrt{\frac{n}{n-2}} C^{-1} \exp \left(\frac{-\left\| \Phi(x , y) \right\|}{n\sqrt{n}}\right).
$$
 \end{lemma}
  \begin{proof}
  Choose $(i , j)$ according to Lemma \ref{Tu}. First assume that 
  $$
 0 < \frac{(\alpha_{n} - \alpha_{i}) (\frac{x}{y} - \alpha_{j})}{(\alpha_{n} -\alpha_{j})( \frac{x}{y}- \alpha_{i} )} \in \mathbb{R}.
  $$
  Then 
  $$
  \left|\log \frac{(\alpha_{n} - \alpha_{i}) (\frac{x}{y} - \alpha_{j})}{(\alpha_{n} -\alpha_{j})( \frac{x}{y}- \alpha_{i} )}\right| = |T_{i, j}(x , y)|,
  $$
  and the proof follows immediately from Lemma \ref{Tu}. Now assume that
  $$
  \frac{(\alpha_{n} - \alpha_{i}) (\frac{x}{y} - \alpha_{j})}{(\alpha_{n} -\alpha_{j})( \frac{x}{y}- \alpha_{i} )} =  R \exp(\mathbf{i}\, \theta),
  $$
  where  $\theta$ is the principal argument of $\frac{(\alpha_{n} - \alpha_{i}) (\frac{x}{y} - \alpha_{j})}{(\alpha_{n} -\alpha_{j})( \frac{x}{y}- \alpha_{i} )}$ and   $\mathbf{i}^2 = -1$. Then 
  \begin{equation}\label{RTij}
  \log R = T_{i, j}.
  \end{equation}
   Next we will estimate the argument $\theta$. Let $\theta_{1}$ and $\theta_{2}$ be the principle arguments of $\frac{\alpha_{n} - \alpha_{i}}{ \frac{x}{y}- \alpha_{i} }$  and $\frac{\frac{x}{y} - \alpha_{j}}{\alpha_{n} -\alpha_{j}}$. We have 
   $$
   |\theta| \leq |\theta_{1}| + |\theta_{2}|.
   $$
   To estimate $|\theta_{1}|$, we notice that the argument of $\frac{ \frac{x}{y}- \alpha_{i} }{\alpha_{n} - \alpha_{i}}$ is $- \theta_{1}$. By \eqref{lessthanexp} and Lemma \ref{mahl5}, we have
   \begin{eqnarray*}
& &   \left| 1 -  \frac{ \frac{x}{y}- \alpha_{i}} {\alpha_{n}- \alpha_{i} }\right|\\
& =& \left|\frac{\frac{x}{y} - \alpha_{n}}{\alpha_{n}- \alpha_{i} }\right|\\
& < & \exp \left(\frac{-\left\| \Phi(x , y) \right\|}{n\sqrt{n}}\right) \frac{   |D|^{\frac{1}{n^2 (n-2)}} M(F)^{\frac{2n -2}{n(n-2)}}|F(x , y)|^{1/n}}{y\,  |\alpha_{n}- \alpha_{i}| } \\
& \leq & \exp \left(\frac{-\left\| \Phi(x , y) \right\|}{n\sqrt{n}}\right) \frac{   |D|^{\frac{1}{n^2 (n-2)}} M(F)^{\frac{2n -2}{n(n-2)}}|F(x , y)|^{1/n}(n+1)^{n} M(F)^{n-1}}{y\,   \sqrt{3}  }\\
& \leq & C^{-1} \exp \left(\frac{-\left\| \Phi(x , y) \right\|}{n\sqrt{n}}\right) \, \frac{1}{2 \sqrt{n}},
\end{eqnarray*}  
where the last inequality is from our assumption on the size of $y$.

From elementary trigonometry, we have
$$
\sin |\theta_{1}| \leq  \left| 1 -  \frac{ \frac{x}{y}- \alpha_{i}} {\alpha_{n}- \alpha_{i} }\right|.
$$
Therefore, from the Taylor expansion of $\sin(z)$ about $0$, we conclude that
$$
|\theta_{1}| <\frac{6}{5} C^{-1} \exp \left(\frac{-\left\| \Phi(x , y) \right\|}{n\sqrt{n}}\right) \, \frac{1}{2 \sqrt{n}}.
$$
Similarly, we have
$$
|\theta_{2}| < \frac{6}{5} C^{-1} \exp \left(\frac{-\left\| \Phi(x , y) \right\|}{n\sqrt{n}}\right) \, \frac{1}{2 \sqrt{n}}.
$$
Since $n \geq 3$, we have
$$
|\theta| \leq |\theta_{1}| + |\theta_{2}| <  C^{-1} \exp \left(\frac{-\left\| \Phi(x , y) \right\|}{n\sqrt{n}}\right).
$$
This, together with \eqref{RTij} and Lemma \ref{Tu}, implies that
$$
 \left|\log \frac{(\alpha_{n} - \alpha_{i}) (\frac{x}{y} - \alpha_{j})}{(\alpha_{n} -\alpha_{j})( \frac{x}{y}- \alpha_{i} )}\right| \leq \sqrt{\frac{n}{n-2}} C^{-1} \exp \left(\frac{-\left\| \Phi(x , y) \right\|}{n\sqrt{n}}\right).
$$
 \end{proof}

\section{the Exponential Gap Principle}\label{EGP}

\begin{thm}\label{exg}
Fix a positive number $C$. Suppose $n \geq 5$. Let  $(x_{1} , y_{1})$, $(x_{2} , y_{2})$ and $(x_{3} , y_{3})$ be three integral solutions to the inequality $|F(x , y)| \leq m$ that are all related to the real root $\alpha_{n}$ of $F(X , 1) = 0$,   with 
\begin{equation}\label{Ass}
   y_{j} \geq C\frac{2}{\sqrt{3}} (n+1)^{n}  \sqrt{n} \,  m^{n}  |D|^{\frac{1}{n^2 (n-2)}} M(F)^{\frac{2n -2}{n(n-2)} + n-1}, 
  \end{equation}
for $j =1, 2, 3$.    Set  $\mathbf{r}_{j} = \left\| \Phi(x_{j} , y_{j}) \right\|$.
If $\mathbf{r}_{1} \leq \mathbf{r}_{2} \leq \mathbf{r}_{3}$ then 
\begin{eqnarray*}
\mathbf{r}_{3} >
A\,  \exp\left( \frac{ \mathbf{r}_{1}}{n\sqrt{n}}\right),
\end{eqnarray*}
where $$A = \frac{C}{4}\, \left((n-1)^2 + n-1\right)^{3/4} \left[(n-2) (n-1) \, \log m +   (n-2)^2 \log \left(  M(F) \right)\right]^{3/2}.$$
\end{thm}
\begin{proof}
We define the vectors ${\bf \vec{e}}(i , j)$ ($i, j \in \{1, 2, 3\}$) as follows.
\begin{eqnarray}\nonumber
\label{defeij}
  {\bf \vec{e}}(i, j)& := & \Phi(x_{i} , y_{i}) - \Phi(x_{j} , y_{j}) \\
&=& \left(\log \left|\frac{(x_{i} - \alpha_{1}y_{i}) F(x_{j}, y_{j})^{1/n}}{(x_{j} - \alpha_{1}y_{j})F(x_{i}, y_{i})^{1/n}}\right|, \ldots, \log \left|\frac{(x_{i} - \alpha_{n}y_{i}) F(x_{j}, y_{j})^{1/n}}{(x_{j} - \alpha_{n}y_{j})F(x_{i}, y_{i})^{1/n}}\right| \right).
\end{eqnarray}
We also define
 $$c: =   \| {\bf \vec{e}} (3 , 1) \|, \,  b:=  \| {\bf \vec{e}} (3 , 2) \|  \, \textrm{and } a :=  \| {\bf \vec{e}} (2 , 1) \|,$$
 so that $$ a \leq b \leq c.$$
 In \eqref{a+b-cF}, we will show that $c -a-b \neq 0$. This implies that the three points  $\Phi(x_{j} , y_{j})$  ($j=1, 2, 3$) are not collinear.

Let $\Delta$ be the triangle with vertices $\Phi_{1} = \Phi(x_{1}, y_{1})$,  $\Phi_{2} = \Phi(x_{2}, y_{2})$ and  $\Phi_{3} =\Phi(x_{3}, y_{3})$.  The length of each side of $\Delta$ is less than $2\mathbf{r}_{3}$. 
By Lemma \ref{gp1} and some elementary geometry,  the shortest altitude of triangle  $\Delta$ has length at most 
$$
\frac{2}{C} \, \exp\left(\frac{ -\mathbf{r}_{1}}{n\sqrt{n}}\right).
$$
Therefore, the area of $\Delta$ is less than 
\begin{equation}\label{up}
\frac{2}{C}\,   \mathbf{r}_{3} \exp\left(\frac{- \mathbf{r}_{1}}{n\sqrt{n}}\right).
\end{equation}

In order to estimate the area of $\Delta$ from below, we will estimate the length of  each side of 
$\Delta$. We proceed to show
\begin{equation}\label{proc}
 \| {\bf \vec{e}} (i , j) \|  >  \sqrt{(n-1)^2 + n-1}\left[ (n-2) (n-1) \log m +  (n-2)^2 \log \left(  M(F) \right)\right]
 \end{equation}
 Using the triangle inequality and our assumption that 
  $(x_{i}, y_{i})$ is related to $\alpha_{n}$, we have for $l = 1, \ldots, n-1$,
 \begin{equation}\label{nlli}
 |\alpha_{n} - \alpha_{l}| \leq | \frac{x_{i}}{y_{i}} - \alpha_{l}|  + | \frac{x_{i}}{y_{i} }- \alpha_{n}| \leq 2  | \frac{x_{i}}{y_{i}} - \alpha_{l}|.
 \end{equation}
 Similarly,
 $$
 |\alpha_{n} - \alpha_{l}| \leq 2  | \frac{x_{j}}{y_{j}} - \alpha_{l}|.
 $$
 Also by Lemmata \ref{3S} and \ref{mahl5}, and since 
$$y_{i}, y_{j} \geq \frac{2}{\sqrt{3}} (n+1)^{n}  \sqrt{n} \,  m^{1/n}  |D|^{\frac{1}{n^2 (n-2)}} M(F)^{\frac{2n -2}{n(n-2)} + n-1}, $$ 
we have
\begin{equation}\label{triineq}
 | \frac{x_{i}}{y_{i}} - \alpha_{n}| , | \frac{x_{j}}{y_{j}} - \alpha_{n}|  < \frac{ |\alpha_{n} - \alpha_{l}|}{2}.
 \end{equation}
 Now we can estimate the size of each coordinate of the vector  $\| {\bf \vec{e}} (i , j) \|$. We have, for $l \in \{1, \ldots, n-1\}$, 
  \begin{eqnarray}\label{cordin}
  \left|\frac{(x_{i} - \alpha_{l}y_{i}) F(x_{j}, y_{j})^{1/n}}{(x_{j} - \alpha_{l}y_{j})F(x_{i}, y_{i})^{1/n}}\right| = \left| \left(\frac{ y_{i}}{y_{j}}\right) \frac{(\frac{x_{i}}{y_{i} }- \alpha_{l}) F(x_{j}, y_{j})^{1/n}}{(\frac{x_{j}}{y_{j}} - \alpha_{l})F(x_{i}, y_{i})^{1/n}}\right|.
\end{eqnarray}
From \eqref{triineq} and \eqref{nlli}, and the fact that $1 \leq |F(x_{i} , y_{i})| \leq m$, we get
\begin{eqnarray}\label{theotherpart}
\left|  \frac{(\frac{x_{i}}{y_{i} }- \alpha_{l}) F(x_{j}, y_{j})^{1/n}}{(\frac{x_{j}}{y_{j}} - \alpha_{l})F(x_{i}, y_{i})^{1/n}}\right| & \leq& \left| \frac{(\frac{x_{i}}{y_{i} }- \alpha_{l}) m^{1/n}}{(\frac{x_{j}}{y_{j}} - \alpha_{l})}\right|\\\nonumber
& \leq& m^{1/n}  \frac{|\frac {x_{i}}{y_{i} }- \alpha_{n}| + |\alpha_{n} - \alpha_{l}|  }{|\frac{x_{j}}{y_{j}} - \alpha_{l}|}\\ \nonumber
& \leq& m^{1/n}   \frac{\frac{|\alpha_{l}- \alpha_{n}|}{2} + |\alpha_{n} - \alpha_{l}|  }{\frac{|\alpha_{l}- \alpha_{n}|}{2}}\\ \nonumber
& \leq& 3 m^{1/n}.
 \end{eqnarray}
Similarly, 
\begin{equation}\label{theotherpart2}
\left|  \frac{(\frac{x_{i}}{y_{i} }- \alpha_{l}) F(x_{j}, y_{j})^{1/n}}{(\frac{x_{j}}{y_{j}} - \alpha_{l})F(x_{i}, y_{i})^{1/n}}\right| \geq \left(3 m^{1/n}\right)^{-1}.
\end{equation}
Therefore, by \eqref{cordin}, 
 \begin{eqnarray}\label{L+U}
\frac{ \left(\frac{ y_{i}}{y_{j}}\right) } {3 m^{1/n}}
 \leq  \left|\frac{(x_{i} - \alpha_{l}y_{i}) F(x_{j}, y_{j})^{1/n}}{(x_{j} - \alpha_{l}y_{j})F(x_{i}, y_{i})^{1/n}}\right| \leq  3 m^{1/n} \left(\frac{ y_{i}}{y_{j}}\right).
\end{eqnarray}
Assume, without loss of generality, that  $|y_{i}| > |y_{j}|$ (notice that $ \| {\bf \vec{e}} (i , j) \| = \| {\bf \vec{e}} (j , i) \|$). By 
\eqref{S65},
 $|y_{i}| \geq \frac{|y_{j}|^{n-1}}{M^{n-2}}$. This, together with our assumption \eqref{Ass}, implies that 
 $$ 
 \log  \left| \frac{1}{3}m^{-1/n} \left(\frac{ y_{i}}{y_{j}}\right) \right| >0,  \, \textrm{and}\,  \,  
 \log  \left| 3  m^{1/n} \left(\frac{ y_{i}}{y_{j}}\right) \right|  > 0.
 $$
We also have
 \begin{eqnarray}\label{n-1number}
& &  \left|\frac{(x_{i} - \alpha_{n}y_{i}) F(x_{j}, y_{j})^{1/n}}{(x_{j} - \alpha_{n}y_{j})F(x_{i}, y_{i})^{1/n}}\right| \\ \nonumber
&=  &  \left|\frac {F(x_{i}, y_{i})^{(n-1)/n} \prod_{l\neq n} (x_{j} - \alpha_{l}y_{j})}{F(x_{j}, y_{j})^{(n-1)/n} \prod_{l \neq n}(x_{i} - \alpha_{l}y_{i})}\right|.
\end{eqnarray}
Therefore, by \eqref{L+U}, we obtain
\begin{eqnarray}\label{byL+U}
& & m^{(1-n)/n} 3^{-(n-1)} \left(\frac{y_{j}}{y_{i}} \right)^{n-1}\\ \nonumber
&\leq& \left|\frac{(x_{i} - \alpha_{n}y_{i}) F(x_{j}, y_{j})^{1/n}}{(x_{j} - \alpha_{n}y_{j})F(x_{i}, y_{i})^{1/n}}\right| \\ \nonumber
&\leq& m^{(n-1)/n} 3^{n-1} \left(\frac{y_{j}}{y_{i}}\right)^{n-1}.
\end{eqnarray}
Since we assumed $y_{i} > y_{j}$, by \eqref{S65} and  \eqref{Ass}, both positive numbers
 $$m^{(1-n)/n} 3^{-(n-1)} \left(\frac{y_{j}}{y_{i}}\right)^{n-1}$$ and $$m^{(n-1)/n} 3^{n-1} \left(\frac{y_{j}}{y_{i}}\right)^{n-1}$$ are smaller than $1$.
We conclude
that 
\begin{eqnarray*}
  (n-1) \log  \left(\frac{y_{i}}{ 3\, m^{\frac{1}{n}} \,  y_{j}}\right)
&\leq& \left| \log   \left|\frac{(x_{i} - \alpha_{n}y_{i}) F(x_{j}, y_{j})^{1/n}}{(x_{j} - \alpha_{n}y_{j})F(x_{i}, y_{i})^{1/n}}\right|  \right|\\
&\leq&   (n-1)\left[ \log\left(\frac{y_{i}}{ 3\, m^{\frac{1}{n}} \,  y_{j}}\right) + 2\, \log \left(3\, m^{\frac{1}{n}}\right) \right].
\end{eqnarray*}
By \eqref{L+U}, for $\alpha_{l} \neq \alpha_{n}$, we have
$$
 \log  \left(\frac{y_{i}}{ 3\, m^{\frac{1}{n}} \,  y_{j}}\right)
\leq\log   \left|\frac{(x_{i} - \alpha_{l}y_{i}) F(x_{j}, y_{j})^{1/n}}{(x_{j} - \alpha_{l}y_{j})F(x_{i}, y_{i})^{1/n}}\right|  \leq \log\left(\frac{y_{i}}{ 3\, m^{\frac{1}{n}} \,  y_{j}}\right) + 2\, \log \left(3\, m^{\frac{1}{n}}\right).
$$
The lower bound \eqref{proc} follows from \eqref{S65} and  \eqref{Ass}.

In order to estimate the area of $\Delta$ from below,  we are going to use the 
Heron's formula, which states that if  $a$, $b$ and $c$ are the lengths  of the sides of a triangle then the area of this triangle is given by
\begin{equation}\label{HER}
\sqrt{\left(\frac{a + b+ c}{2}\right)  \left(\frac{a + b-c}{2}\right)  \left(\frac{a -b+ c}{2}\right)  \left(\frac{-a + b+ c}{2}\right)}.
\end{equation}
Let ${\bf \vec{e}} (i , j)$ be the vectors that are defined in \eqref{defeij} and
 $$c =   \| {\bf \vec{e}} (3 , 1) \|, \,  b=  \| {\bf \vec{e}} (3 , 2) \|  \, \textrm{and } a =  \| {\bf \vec{e}} (2 , 1) \|,$$
 so that $$ a \leq b \leq c.$$
 By \eqref{proc}, we have
\begin{eqnarray}\label{a+b+c2}
&&\frac{a + b+ c}{2} > \\ \nonumber
 &&\frac{3}{2}  \sqrt{(n-1)^2 + n-1}\left[ (n-2) (n-1) \, \log m +   (n-2)^2 \log \left(  M(F) \right)\right],
\end{eqnarray}
\begin{eqnarray}\label{a-b+c2}
& &\frac{a -b+ c}{2} \geq 
\frac{a}{2} >\\  \nonumber
 & &\frac{1}{2}  \sqrt{(n-1)^2 + n-1}\left[(n-2) (n -1) \, \log m +  (n-2)^2 \log \left(  M(F) \right)\right],
\end{eqnarray}
and
\begin{eqnarray}\label{-a+b+c2}
&&\frac{-a +b+ c}{2} \geq  \frac{c}{ 2} >\\ \nonumber
&&
 \frac{1}{2}  \sqrt{(n-1)^2 + n-1}\left[(n-2) (n-1) \,\log m +  (n-2)^2 \log \left(  M(F) \right)\right].
\end{eqnarray}
Next we will obtain a lower bound for $\frac{a + b-c}{2}$.
We define, for $1 \leq j < i\leq 3$, 
$$
\eta_{i, j}(\alpha_{l}) : = \frac{\log   \left|\frac{(\frac{x_{i}}{y_{i}} - \alpha_{l}) F(x_{j}, y_{j})^{1/n}}{(\frac{x_{j}}{y_{j}} - \alpha_{l})F(x_{i}, y_{i})^{1/n}} 9\, m^{\frac{2}{n}}\right|} { \log  \left(3\, m^{\frac{1}{n}}\right)}
$$
if $l \neq n$, and
$$
\eta_{i, j}(\alpha_{n}) : = \frac{\log   \left|\frac{(\frac{x_{i}}{y_{i}} - \alpha_{n}) F(x_{j}, y_{j})^{1/n}}{(\frac{x_{j}}{y_{j}} - \alpha_{n})F(x_{i}, y_{i})^{1/n}} \left(3\, m^{\frac{1}{n}}\right)^{2(n-1)}\right|} { (n-1)\log  \left(3\, m^{\frac{1}{n}}\right)},
$$
so that, by \eqref{L+U} and \eqref{n-1number},
\begin{equation}\label{sizeofeta}
1\leq \eta_{i, j}(\alpha_{k})\leq 3,
\end{equation}
 for $k= 1, \ldots, n$. 
We also define
\begin{equation*}
a' := a \left(\log \frac{y_{2}}{ 9\, m^{\frac{2}{n}}\, y_{1}}\right)^{-1},
\end{equation*}
\begin{equation*}
b' :=  b \left(\log \frac{y_{3}}{ 9\, m^{\frac{2}{n}} y_{2}}\right)^{-1},
\end{equation*}
and 
\begin{equation*}
c' := c \left(\log \frac{y_{3}}{ 9\, m^{\frac{2}{n}}y_{1}}\right)^{-1}.
\end{equation*}
Then
\begin{equation}\label{aa}
a' =\sqrt{(n-1)^2\left(1+  \frac{\eta_{2, 1}(\alpha_{n}) \log\left( 3m^{1/n}\right)}{\log \frac{y_{2}}{ 9\, m^{\frac{2}{n}}\, y_{1}}} \right)^2 +
\sum_{l=1}^{n-1}\left(1+  \frac{\eta_{2, 1}(\alpha_{l}) \log \left( 3m^{1/n}\right)} {\log \frac{y_{2}}{ 9\, m^{\frac{2}{n}}\, y_{1}}}\right)^2},
\end{equation}
\begin{equation}\label{bb}
b' =\sqrt{(n-1)^2\left(1+  \frac{\eta_{3, 2}(\alpha_{n}) \log\left( 3m^{1/n}\right)}{\log \frac{y_{3}}{ 9\, m^{\frac{2}{n}}\, y_{2}}} \right)^2 +
\sum_{l=1}^{n-1}\left(1+  \frac{\eta_{3, 2}(\alpha_{l}) \log \left( 3m^{1/n}\right) }{\log \frac{y_{3}}{ 9\, m^{\frac{2}{n}}\, y_{2}}}\right)^2},
\end{equation}
and
\begin{equation}\label{cc}
c' =\sqrt{(n-1)^2\left(1+ \frac{ \eta_{3, 1}(\alpha_{n}) \log\left( 3m^{1/n}\right)}{\log \frac{y_{3}}{ 9\, m^{\frac{2}{n}}\, y_{1}}} \right)^2 +
\sum_{l=1}^{n-1}\left(1+ \frac{ \eta_{3, 1}(\alpha_{l}) \log \left( 3m^{1/n}\right)}{\log \frac{y_{3}}{ 9\, m^{\frac{2}{n}}\, y_{1}}} \right)^2}.
\end{equation}

To give a lower bound for $a+ b -c$, we rearrange this summation as
$$
a+ b -c = a - c_{1} + b - c_{2},
$$
where $c_{1} =    \frac{   \log \left(\frac{y_{2}}{ y_{1}}\right) }{\log \left(\frac{y_{3}}{ 9\, m^{\frac{2}{n}}y_{1}}\right)}\, \,  c$, $c_{2} =  \frac{   \log \left(\frac{y_{3}}{ 9\, m^{\frac{2}{n}}y_{2}}\right) }{\log \left(\frac{y_{3}}{ 9\, m^{\frac{2}{n}}y_{1}}\right)}\, \,  c $, and $c_{1} + c_{2} = c$. By \eqref{bb},  \eqref{cc} and \eqref{sizeofeta},
we have 
\begin{eqnarray}\label{coverb}
\frac{c}{b} & = & \frac{   \log \left(\frac{y_{3}}{ 9\, m^{\frac{2}{n}}y_{1}}\right) }{\log \left(\frac{y_{3}}{9\, m^{\frac{2}{n}} y_{2}}\right)} \frac{c'}{b'} \\ \nonumber
&\leq& \frac{   \log \left(\frac{y_{3}}{ 9\, m^{\frac{2}{n}}y_{1}}\right) }{\log \left(\frac{y_{3}}{ 9\, m^{\frac{2}{n}}y_{2}}\right)} \times \frac{\sqrt{(n-1)^2 + (n -1)} \, \left(1 +  \frac{ 3 \log \left(3\, m^{\frac{1}{n}}\right)}{\log \left(\frac{y_{3}}{ 9\, m^{\frac{2}{n}}\, y_{1}}\right)} \right)} {\sqrt{(n-1)^2 + (n -1)} \, \left(1 +  \frac{ 1 \log \left(3\, m^{\frac{1}{n}}\right)}{\log \left(\frac{y_{3}}{ 9\, m^{\frac{2}{n}}\, y_{2}}\right)} \right)}\\ \nonumber
&\leq& \frac{   \log \left(\frac{y_{3}}{ 9\, m^{\frac{2}{n}}y_{1}}\right) }{\log \left(\frac{y_{3}}{ 9\, m^{\frac{2}{n}}y_{2}}\right)} \, \left(1 +  \frac{ 3 \log \left(3\, m^{\frac{1}{n}}\right)}{\log \left(\frac{y_{3}}{ 9\, m^{\frac{2}{n}}\, y_{1}}\right)} \right)
\end{eqnarray}
Therefore,
$$
c_{2} = \frac{   \log \left(\frac{y_{3}}{ 9\, m^{\frac{2}{n}}y_{2}}\right) }{\log \left(\frac{y_{3}}{ 9\, m^{\frac{2}{n}}y_{1}}\right)}\, \,  c < \left(1 +  \frac{ 3 \log \left(3\, m^{\frac{1}{n}}\right)}{\log \left(\frac{y_{3}}{ 9\, m^{\frac{2}{n}}\, y_{1}}\right)} \right) b.
$$
This implies that
$$b - c_{2} = b - \frac{   \log \left(\frac{y_{3}}{ 9\, m^{\frac{2}{n}}y_{2}}\right) }{\log \left(\frac{y_{3}}{ 9\, m^{\frac{2}{n}}y_{1}}\right)}\, \,  c >  -  \, \frac{ 3 \log \left(3\, m^{\frac{1}{n}}\right)}{\log \left(\frac{y_{3}}{ 9\, m^{\frac{2}{n}}\, y_{1}}\right)} b.
$$
We have shown that  if $b-c_{2}$ is a negative number then it has a relatively small absolute value. Next we will show that $a - c_{1}$ is large enough to cancel this possibly negative value and provide a positive lower bound for $a + b - c  =  a - c_{1} + b - c_{2}$. 
We have
\begin{eqnarray}\label{a+b-c}
a + b - c & = & a - c_{1} + b - c_{2} \\ \nonumber
&> & a - \frac{   \log \left(\frac{y_{2}}{ 
y_{1}}\right) }{\log \left(\frac{y_{3}}{ 9\, m^{\frac{2}{n}}y_{1}}\right)}\, c -  \frac{ 3 \log \left(3\, m^{\frac{1}{n}}\right)}{\log \left(\frac{y_{3}}{ 9\, m^{\frac{2}{n}}\, y_{1}}\right)} b\\ \nonumber
 &> &  a - \left(\frac{   \log \left(\frac{y_{2}}{ y_{1}}\right) }{\log \left(\frac{y_{3}}{ 9\, m^{\frac{2}{n}}y_{1}}\right)}+ \frac{ 3 \log \left(3\, m^{\frac{1}{n}}\right)}{\log \left(\frac{y_{3}}{ 9\, m^{\frac{2}{n}}\, y_{1}}\right)}\right)\,   c\\ \nonumber
 & =& a - \left(\frac{   \log \left(\frac{y_{2}}{ y_{1}}\right) }{\log \left(\frac{y_{3}}{ 9\, m^{\frac{2}{n}}y_{1}}\right)}\right)\,  (1 + \epsilon) \, c\\ \nonumber
 & = &  \log \left(\frac{y_{2}} { 9\, m^{\frac{2}{n}}y_{1}}\right) \left( a' -  (1 + \epsilon) (1 + \delta) \, c' \right).
\end{eqnarray}
where $\epsilon =  \frac{ 3 \log \left(3\, m^{\frac{1}{n}}\right)}{\log \left(\frac{y_{2}}{  y_{1}}\right)}$ and $\delta = \frac{ 2 \log \left(3\, m^{\frac{1}{n}}\right)}{\log \left(\frac{y_{2}}{ 9\, m^{\frac{2}{n}}\, y_{1}}\right)}$. By \eqref{Ass} and \eqref{S65},
\begin{equation}\label{edl}
0 <\epsilon, \delta  < \frac{1}{n^2}. 
\end{equation}

Now, we will estimate $a' -  (1 + \epsilon) \, (1 + \delta)\,  c' $.
By \eqref{sizeofeta}, we have
\begin{eqnarray*}
& & \sqrt{(n-1)^2 + (n -1)} \, \left(1 +  \frac{ \log \left(3\, m^{\frac{1}{n}}\right)}{\log \frac{y_{2}}{ 9\, m^{\frac{2}{n}}\, y_{1}}} \right)\\
& & < a' < \\
& & \sqrt{(n-1)^2 + (n -1)} \, \left(1 +  \frac{ 3 \log \left(3\, m^{\frac{1}{n}}\right)}{\log \frac{y_{2}}{ 9\, m^{\frac{2}{n}}\, y_{1}}} \right)
\end{eqnarray*}
and
\begin{eqnarray*}
& &\sqrt{(n-1)^2 + (n -1)} \, \left(1 +  \frac{ \log \left(3\, m^{\frac{1}{n}}\right)}{\log \frac{y_{3}}{ 9\, m^{\frac{2}{n}}\, y_{1}}} \right)\\ 
&  & < c' <\\
& & \sqrt{(n-1)^2 + (n -1)} \, \left(1 +  \frac{ 3 \log \left(3\, m^{\frac{1}{n}}\right)}{\log \frac{y_{3}}{ 9\, m^{\frac{2}{n}}\, y_{1}}} \right)
\end{eqnarray*}
Therefore,
\begin{eqnarray}\label{epsilonanddelta}
& & a' -  (1 + \epsilon) \, (1 + \delta)\, c' \\ \nonumber
&\geq& \sqrt{(n-1)^2 + (n -1)} \, \left(  \frac{ \log \left(3\, m^{\frac{1}{n}}\right)}{\log \frac{y_{2}}{ 9\, m^{\frac{2}{n}}\, y_{1}}}-  \frac{ 3 (1 + \epsilon) \, (1 + \delta)\, \log \left(3\, m^{\frac{1}{n}}\right)}{\log \frac{y_{3}}{ 9\, m^{\frac{2}{n}}\, y_{1}}} \right)\\ \nonumber
&> &   \sqrt{(n-1)^2 + (n -1)} \, \left(   \frac{  \left[n-1 - 3(1+\epsilon) (1+ \delta)\right]\log \left(3\, m^{\frac{1}{n}}\right)}{(n-1)\log \frac{y_{2}}{ 9\, m^{\frac{2}{n}}\, y_{1}}}\right),
\end{eqnarray}
because, by \eqref{Ass}, we have
$$
\frac{y_{3}}{ 9\, m^{\frac{2}{n}}\, y_{1}} \geq \frac{y^{n-1}_{2}}{ 9\, m^{\frac{2}{n}}\, M(F)^2y_{1}}  > \left(\frac{y_{2}}{ 9\, m^{\frac{2}{n}}\, y_{1}} \right)^{n-1}.
$$
Therefore, by \eqref{edl} and \eqref{a+b-c}, when $n \geq 5$, we have
\begin{eqnarray}\label{a+b-cF}
a + b - c &>&   \left[n-1 - 3.26\right]\log \left(3\, m^{\frac{1}{n}}\right)\\ \nonumber
&\geq & 0.74 \,  \log \left(3\, m^{\frac{1}{n}}\right) >  0.74.
\end{eqnarray}
This, together with \eqref{HER}, \eqref{a+b+c2}, \eqref{a-b+c2}, and \eqref{-a+b+c2}, implies that
the area of $\Delta$ is greater than 
$$
\frac{1}{2}  \left((n-1)^2 + n-1\right)^{3/4} \left[(n-2) n \, \log m +  (n-1) (n-2)\log \left(  M(F) \right)\right]^{3/2}.
$$
Comparing this lower bound with the upper bound that is obtained in \eqref{up}, the proof of the theorem is complete.
\end{proof}

\begin{thm}\label{exg3,4}
Suppose $n \geq3$. Let  $(x_{1} , y_{1})$, $(x_{2} , y_{2})$, $(x_{3} , y_{3})$ and $(x_{4} , y_{4})$ be four integral solutions to the inequality $|F(x , y)| \leq m$ that are all related to the real root $\alpha_{n}$ of $F(X , 1) = 0$,   with 
\begin{equation}\label{Ass3, 4}
  y_{j} \geq C\frac{2}{\sqrt{3}} (n+1)^{n}  \sqrt{n} \,  m^{n}  |D|^{\frac{1}{n^2 (n-2)}} M(F)^{\frac{2n -2}{n(n-2)} + n-1}, 
  \end{equation}
for $j =1, 2, 3, 4$ where $C \geq 1$.   Set  $\mathbf{r}_{j} = \left\| \Phi(x_{j} , y_{j}) \right\|$. Suppose that $\mathbf{r}_{1} \leq \mathbf{r}_{2} \leq \mathbf{r}_{3} \leq \mathbf{r}_{4}$. Then
$$
\mathbf{r}_{4} > A\,  \exp\left( \frac{ \mathbf{r}_{1}}{n\sqrt{n}}\right),
$$
where $$A = \frac{C}{8 \sqrt{2}}\, \left((n-1)^2 + n-1\right)^{3/4} \left[(n-2) n \, \log m +  (n-1) (n-2)\log \left(  M(F) \right)\right]^{3/2}.$$
\end{thm}
\begin{proof}
The proof  is almost identical to the proof of Theorem \ref{exg}. The inequality \eqref{a+b-cF} is not useful for $n < 5$ (the right-hand side would be negative). By assuming the existence of $4$ large solutions, as opposed to 3 large solutions in the previous theorem, we can improve \eqref{a+b-cF} to an inequality that works for smaller degrees $3$ and $4$, as well. We will consider the triangle with vertices $\Phi(x_{1} , y_{1})$, $\Phi(x_{2} , y_{2})$ and $\Phi(x_{4} , y_{4})$.
Let $$c: =   \| {\bf \vec{e}} (4 , 1) \|, \,  b:=  \| {\bf \vec{e}} (4 , 2) \|  \, \textrm{and } a :=  \| {\bf \vec{e}} (2 , 1) \|.$$
We also define
\begin{equation*}
a': = a \left(\log \frac{y_{2}}{ 9\, m^{\frac{2}{n}}\, y_{1}}\right)^{-1},
\end{equation*}
\begin{equation*}
b' :=  b \left(\log \frac{y_{4}}{ 9\, m^{\frac{2}{n}} y_{2}}\right)^{-1},
\end{equation*}
and 
\begin{equation*}
c' := c \left(\log \frac{y_{4}}{ 9\, m^{\frac{2}{n}}y_{1}}\right)^{-1}.
\end{equation*}
Similarly to \eqref{a+b-c}, we have
$$
a + b - c > \log \left(\frac{y_{2}} { 9\, m^{\frac{2}{n}}y_{1}}\right) \left( a' -  (1 + \epsilon) (1 + \delta) \, c' \right).
$$
By \eqref{Ass3, 4}, we have
$$
\frac{y_{4}}{ 9\, m^{\frac{2}{n}}\, y_{1}} \geq \frac{y^{n-1}_{3}}{ 9\, m^{\frac{2}{n}}\, M(F)^2y_{1}}  \geq \frac{y^{(n-1)^2}_{2}}{ 9\, m^{\frac{2}{n}}\, M(F)^{2n} y_{1}}> \left(\frac{y_{2}}{ 9\, m^{\frac{2}{n}}\, y_{1}} \right)^{(n-1)^2}.
$$
Therefore, in place of \eqref{epsilonanddelta}, we obtain

\begin{eqnarray*}
& & a' -  (1 + \epsilon) \, (1 + \delta)\, c' \\ \nonumber
&\geq& \sqrt{(n-1)^2 + (n -1)} \, \left(  \frac{ \log \left(3\, m^{\frac{1}{n}}\right)}{\log \frac{y_{2}}{ 9\, m^{\frac{2}{n}}\, y_{1}}}-  \frac{ 3 (1 + \epsilon) \, (1 + \delta)\, \log \left(3\, m^{\frac{1}{n}}\right)}{\log \frac{y_{4}}{ 9\, m^{\frac{2}{n}}\, y_{1}}} \right)\\ \nonumber
&> &   \sqrt{(n-1)^2 + (n -1)} \, \left(   \frac{  \left[(n-1)^2 - 3(1+\epsilon) (1+ \delta)\right]\log \left(3\, m^{\frac{1}{n}}\right)}{(n-1)^2\log \frac{y_{2}}{ 9\, m^{\frac{2}{n}}\, y_{1}}}\right).
\end{eqnarray*}
We conclude that
\begin{eqnarray*}\label{a+b-cF'}
a + b - c >   \left[\frac{14}{81}\right]\frac{\log \left(3\, m^{\frac{1}{n}}\right)}{n-1}.
\end{eqnarray*}
Since $n \leq 4$, we obtain
$$
a + b - c >   \left[\frac{14}{243}\right]\log \left(3\, m^{\frac{1}{n}}\right) \geq \left[\frac{14}{243}\right]\log 3 .
$$
This inequality will be used instead of \eqref{a+b-cF} in this proof.
The rest  of estimates are the same as those in the proof of Theorem \ref{exg}.
\end{proof}

 \section{Application of the theory of Linear forms in logarithms}\label{LFL}
 
 \subsection{Set up}
 
 Let $\alpha_{n}$ be areal root of $F(X, 1) =0$. Suppose that there are three primitive solutions 
 $(x_{1}, y_{1})$, $(x_{2}, y_{2})$, $(x_{3}, y_{3})$ to $|F(x , y)| \leq m$ with
  $$
   y_{j} \geq C\frac{2}{\sqrt{3}} (n+1)^{n}  \sqrt{n} \,  m^{n}  |D|^{\frac{1}{n^2 (n-2)}} M(F)^{\frac{2n -2}{n(n-2)} + n-1}
   $$
 for $j=1, 2, 3$ related to $\alpha_{n}$, where $C \geq 1$ is a number to be specified later. Define
 $$
 \mathbf{r}_{j} := \left\| \Phi( x_{j} , y_{j}) \right\|
$$
for $j = 1, 2, 3$ and assume that $\mathbf{r}_{1} \leq \mathbf{r}_{2} \leq \mathbf{r}_{3}$. Let $i, j$ be indices chosen according to Lemma  \ref{Tu}. We will later specify a finite set of places $S$ of the number field $\mathbb{Q}(\alpha_{j})$, containing all infinite places of this field, such that $x - \alpha_{j} y$ is an $S$-unit of $\mathbb{Q}(\alpha_{j})$ for all solutions $(x , y)$ of $|F(x , y)| \leq m$ under consideration. Let $\{\epsilon_{1}, \ldots, \epsilon_{s-1}\}$ a  fundamental system of  $S$-units. Then in particular, $$
x_{3} - \alpha_{j} y_{3} =  \zeta \epsilon_{1}^{b_{1}} \ldots \epsilon_{s-1}^{b_{s-1}},
$$
 for some root of unity $\zeta$ and rational integers $b_{1}$, \ldots, $b_{s-1}$. Let $\sigma$ be the $\mathbb{Q}$-isomorphism from $\mathbb{Q}(\alpha_{j})$ to $\mathbb{Q}(\alpha_{i})$ suh that $\sigma(\alpha_{j}) = \alpha_{i}$ and put $\epsilon'_{k}: = \sigma(\epsilon_{k})$ for $k= 1, \ldots, s-1$. Then
  \begin{eqnarray} \nonumber
   & & \left|  \log \left(\frac{(\alpha_{n} - \alpha_{i}) (x_{3} - \alpha_{j}y_{3})}{(\alpha_{n} -\alpha_{j})(x_{3} - \alpha_{i}y_{3})} \right) \right|\\ \nonumber
  & = &   \left| \log \left(\frac{(\alpha_{n} - \alpha_{i}) (t_{3} - \alpha_{j})}{(\alpha_{n} -\alpha_{j})(t_{3} - \alpha_{i})} \right)\right| \\  \label{bkinrep}
   & = &  \left| \log \lambda_{i,j} + \sum_{k=2}^{s} b_{k}\log  \lambda_{k} + 2 w \pi \mathbf{i} \right|,
  \end{eqnarray}
where $t_{3} =  \frac{x_{3}}{y_{3}}$, $\lambda_{i, j} = \frac{\alpha_{n} - \alpha_{i}}{\alpha_{n} -\alpha_{j}}$ and
 $\lambda_{k+1} = \frac{\epsilon_{k}}{ \epsilon'_{k}}$ for $k = 2, \ldots, s-1$ and $w$ is a rational integer.

We will apply Proposition \ref{mat} to obtain a lower bound for 
$$ \left|  \log \left(\frac{(\alpha_{n} - \alpha_{i}) (x_{3} - \alpha_{j}y_{3})}{(\alpha_{n} -\alpha_{j})(x_{3} - \alpha_{i}y_{3})} \right) \right|.$$
We will work in the number field $\mathbb{Q}(\alpha_{n}, \alpha_{i} , \alpha_{j})$ of degree $d$. Trivially
$$
d \leq n(n-1) (n-2).
$$ 
We will  find appropriate values for  the quantities $A_{k}$ and $B$ in Proposition \ref{mat}.
It turns out that in the statement of Proposition \ref{mat}, we may take 
\begin{eqnarray*}
N & = & |S| + 1 = s + 1.
\end{eqnarray*}
These values for $N$ and $d$ imply the following values for $C(|S|)$ and $C_{0}$:
\begin{eqnarray*}
\, \, \, \, \quad  C(|S|) &=& \frac{150}{ (s)!} (s+ 2)^{s + 5}  4^{s+1}  \exp (s+1),\\
\, \, \, \,C_{0} &=& 10 (s+1).
\end{eqnarray*}
In Subsections \ref{subAk}, \ref{subA1} and \ref{subB}, we find appropriate values for $A_{k}$'s and $B$ in Proposition \ref{mat}. These values will imply the following:
\begin{eqnarray}\label{W00}
 \qquad  \quad W_{0}  &= &4\, \log(3d)+ 2 \log (s-1)! + \log  \left( \mathbf{r}_{3} +  \log m   \right)+ \log (s-1)\\ \nonumber
 &<& 10\log (s-1)! + \log  \left( \mathbf{r}_{3} + \log m \right),
\end{eqnarray}
and
\begin{eqnarray}\label{Om}
 \Omega  = \frac{ \left(2 \pi \,  n (n-1) (n-2)\right)^{s}}{(n-1)!}    \left((s -1)!\right)^{2} |D| \left(\log |D|\right)^n (\log m)^{s}
 \times\\ \nonumber
\times \left[ \frac{4}{\sqrt{n}}\bold{r}_{1} + 4 \log m \right]  \left(\frac{\log n} {\log \log n}  \right)^{3(s-1)}.
\end{eqnarray}

One can see that the values  $C(|S|)$ and $\Omega$ are the largest in our estimates above. 
In particular, these two values determine the number of times we need to apply our gap principles established in Section \ref{EGP}.

Once these values are established,  \eqref{bkinrep} and Proposition \ref{mat} imply that
\begin{eqnarray*}
& &\log \left|  \log \left(\frac{(\alpha_{n} - \alpha_{i}) (x_{3} - \alpha_{j}y_{3})}{(\alpha_{n} -\alpha_{j})(x_{3} - \alpha_{i}y_{3})} \right) \right| > -C(|S|) C_{0} W_{0}d^{2}\Omega > \\
&&   -750  (s+ 2)^{s + 6}  4^{s+2}   \left(2 \pi \,  n (n-1) (n-2)\right)^{s} \frac{d^{2}}{s}    (s -1)! |D| \left(\log |D|\right)^n (\log m)^{s}\times\\
& & \times \left[ \frac{4}{\sqrt{n}} \bold{r}_{1}  + 4 \log m\right]    \left(\frac{\log n} {\log \log n}  \right)^{3(s-1)}W_{0}.
\end{eqnarray*}
By inserting Stirling's formula
$(s - 1)! \leq e (s-1)^{s-1/2} e^{-s+1}$, we have
\begin{eqnarray}\label{almightyEs}
 && \log \left|  \log \left(\frac{(\alpha_{n} - \alpha_{i}) (x_{3} - \alpha_{j}y_{3})}{(\alpha_{n} -\alpha_{j})(x_{3} - \alpha_{i}y_{3})} \right) \right| > \\ \nonumber
 &&  -750  (s+ 2)^{2s + 7}  2^{s+2}   \left(2 \pi \,  n (n-1) (n-2)\right)^{s} d^{2}  |D| \left(\log |D|\right)^n (\log m)^{s} \times \\  \nonumber
 && \times  \left[ \frac{4}{\sqrt{n}} \bold{r}_{1}  + 4\, \log m\right]  \log  \left( \mathbf{r}_{3} +  \log m \right)   \left(\frac{\log n} {\log \log n}  \right)^{3(s-1)}.
\end{eqnarray}

\subsection{Estimating $A_{k}$'s}\label{subAk}
  
  Let $\mathbb{K}$ be an algebraic number field of degree $d_{1}$ with
    unit rank $r$. 
     Let $S$ be a finite set of places on $\mathbb{K}$ containing the set of infinite places $S_{\infty}$. Let $t$ be the number of finite places in $S$ and $|S| = s$. Then we have
  $$
  s-1 = r+t.
  $$
We denote the $S$-regulator of $\mathbb{K}$ by $R_{S}$. 
Let $P$ be the maximum of the norms of the prime ideals corresponding to  non-Archimedean places in $S$. Let $D_{\mathbb{K}}$ be the discriminant of $\mathbb{K}$ and $q$ be the number of complex places of $\mathbb{K}$. Put
\begin{equation}\label{Delta}
\Delta = \left(\frac{2}{\pi}\right)^{q} |D_{\mathbb{K}}|^{1/2}.
\end{equation}
It is shown in \cite{BugG} that if $d_{1}\geq 2$ then 
\begin{equation}\label{E5BugG}
0< R_{S} \leq \Delta \left( \log \Delta\right)^{d_{1}-1-q} \left(d_{1} - 1 + \log \Delta \right)^q \left(d_{1} \log^{*}P\right)^t / (d_{1} -1)! .
\end{equation}

Let
$$
c_{4} = c_{4}(d_{1} , s) = \left( (s - 1)!\right)^2 / (2^{s-2} d_{1}^{s-1})
$$
and 
$$
c_{5} = c_{5}(d_{1}, s, \mathbb{K}) = c_{4} \left(\frac{\delta_{\mathbb{K}}}{d_{1}}   \right)^{2-s},
$$
where $$\delta_{\mathbb{K}} = \frac{1}{53d_{1} \log 6d_{1}}.$$
The following is Lemma 1 of \cite{BugG} (see also \cite{Hjd}).
\begin{prop}[Bugeaud and Gy\H{o}ry]\label{L1BugG}
There exists in $\mathbb{K}$ a fundamental system $\{\epsilon_{1}, \ldots, \epsilon_{s-1}\}$ of $S$-units with the following properties:
\begin{enumerate}
\item $\prod_{i=1}^{s-1} \log h(\epsilon_{i}) \leq c_{4} R_{S}$;
\item  $h(\epsilon_{i}) \leq c_{5} R_{S}, \, \, \quad  \textrm{for}\, \,  i= 1, \ldots, s-1$.
\item The absolute values of the entries of the inverse  matrix  of 
$$\left( \log |\epsilon_{i}|_{\nu_{j}}\ \right)_{i, j = 1, \ldots s -1}$$ do not exceed 
$$
\left[ (s-1)! \right]^2 \frac{53\, d_{1} \log 6d_{1}}{2^{s-2}}.
$$
\end{enumerate}
\end{prop}

Let again $F(x , y) \in \mathbb{Z}[x , y]$ be an irreducible binary form of degree $n \geq 3$, and let $a_{0} := F(1 , 0)$. Let as before $\alpha_{1}$, \ldots, $\alpha_{n}$ be the roots in $\mathbb{C}$ of $F(X , 1)$ and $\mathbb{K} := \mathbb{Q}(\alpha_{1})$. We assume that among these there are precisely $r_{1}$ real ones, and $2r_{2}$ non-real ones. Let $\alpha_{i_1}$, \ldots, $\alpha_{i_n}$ be a permutation of $\alpha_{1}$, \ldots, $\alpha_{n}$  such that $\alpha_{i_j} \in \mathbb{R}$ for $j=1, \ldots, r_{1}$ and $\alpha_{i_{j+r_{2}}}=\overline{\alpha_{i_j}}$ for $j=r_{1}+1$, \ldots, $r_{1}+r_{2}$. Then we define the Archimedean valuations on $\mathbb{K}$ by 
\begin{equation}\label{Refereevaluation}
|\gamma|_{\nu_{k}}: = |\gamma_{i_{k}}|^{d_{k}/n}
\end{equation}
for $k = 1, \ldots, r_{1} + r_{2}$, where $\gamma_{i_{k}}$ is the image of 
$\gamma \in \mathbb{K}$ under the embedding given by $\alpha \mapsto \alpha_{i_{k}}$, and where $d_{k} = 1$ for $k = 1, \ldots, r_{1}$, $d_{k} = 2$ for $k = r_{1}+1, \ldots, r_{1}+r_{2}$.

When dealing with the  inequality $|F(x , y)| \leq m$ we assume that $|a_{0}| \leq m$, and when dealing with the equation $|F(x , y)| = m$ we assume that $|a_{0}| = m$. Further, 
when dealing with the  inequality $|F(x , y)| \leq m$ we take the set $S$ to be the set consisting of all Archimedean places of $\mathbb{K}$,   and all  non-Archimedean places of $\mathbb{K}$ lying above the rational primes not exceeding $m$. In this case,
$$
|S| \leq n + n\,  \pi(m),
$$
where $\pi(m)$ is the number of prime numbers less than or equal to the integer $m$.
When dealing with the equation  $|F(x , y)| = m$, we take the set $S$ to be the set consisting of  all Archimedean places of $\mathbb{K}$,   and all  non-Archimedean places of $\mathbb{K}$ lying above the rational prime  divisors of  $m$. In this case,
$$
|S| \leq n + n\,  \omega(m),
$$  
where $\omega(m)$ is the number of prime factors of integer $m$.
In both cases, we have $$d_{1} = n,$$
 $$P \leq m$$ and $$|D_{\mathbb{K}}| \leq |D(F)|.$$
 
  \textbf{Remark.} The algebraic numbers in our linear form in logarithms 
  $$
 \log \lambda_{i,j} + \sum_{k=2}^{s} b_{k}\log  \lambda_{k} + 2 w \pi \mathbf{i} 
   $$
  come from the number field $\mathbb{Q}(\alpha_{n}, \alpha_{i}, \alpha_{j})$ which has degree $d \leq n (n-1) (n-2)$. Therefore, we must use $d$ for the degree of the algebraic number field in Proposition \ref{mat}. However, when estimating the absolute logarithmic heights, we end up working with units in $\mathbb{Q}(\alpha_{n})$ and therefore while applying  Proposition \ref{L1BugG}, we will take $d_{1} = n$.

  In \eqref{bkinrep}, we have $\lambda_{k+1} = \frac{\epsilon_{k}}{\epsilon'_{k}}$ and, by the properties of the logarithmic height, we have
\begin{equation}\label{hless2h}
h (\lambda_{k+1}) \leq 2 h(\epsilon_{k}).
\end{equation}
Further, we have 
$|\log \lambda_{k+1}| \leq |\log \epsilon_{k}| + |\log \epsilon'_{k}|$.
If $\log \epsilon_{k}$ has it principal value, then clearly, by \eqref{defoflogho1}, $|\log \epsilon_{k}| \leq 2n\, h(\epsilon_{k}) + \pi$. By a result of Voutier in \cite{Vou}, we have
$$
n\, h(\epsilon_{k}) \geq \frac{1}{4} \left(\frac{\log \log n}{\log n}  \right)^{3}.
$$
Therefore,
$$
4n\, h(\epsilon_{k}) \left(\frac{\log n} {\log \log n}  \right)^{3} \geq 1$$
and 
$$
|\log \epsilon_{k}| \leq 2n\, h(\epsilon_{k}) + \pi < 4 \pi  n\, h(\epsilon_{k}) \left(\frac{\log n} {\log \log n}  \right)^{3}.
$$
We conclude that 
\begin{equation}\label{noaslog}
|\log \lambda_{k}|  \le  |\log \epsilon_{k}|  + |\log \epsilon'_{k}| < 8  \pi  n\, h(\epsilon_{k}) \left(\frac{\log n} {\log \log n}  \right)^{3}.
\end{equation}
For   the degree $d$ of the number field $\mathbb{Q}(\alpha_{n},  \alpha_{i}, \alpha_{j})$,
 we have $d \leq n (n-1) (n-2)$.  To apply Proposition \ref{mat},  by \eqref{hless2h} and \eqref{noaslog}, we may take  for $k = 2, \ldots, s$,
  $$
  A_{k} = 4 \pi \,  n (n-1) (n-2)   h(\epsilon_{k}) \left(\frac{\log n} {\log \log n}  \right)^{3}.
  $$
 Therefore, by part (1) of Lemma \ref{L1BugG},
    \begin{eqnarray*}
  \prod_{i\geq 2} A_{i} &=& \left[4 \pi  \,  n (n-1) (n-2) \left(\frac{\log n} {\log \log n}  \right)^{3}\right]^{s - 1}  \prod_{i\geq 2} h(\epsilon_{i}) \\ \nonumber
  &\leq&  c_{4} \left[4 \pi \,  n (n-1) (n-2) \left(\frac{\log n} {\log \log n}  \right)^{3}\right]^{s - 1}   R_{S}.
  \end{eqnarray*}
  By  inequality \eqref{E5BugG},  
 $$
 R_{S} \leq \Delta \left( \log \Delta\right)^{n-1-q} \left(n - 1 + \log \Delta \right)^q \left(n \log^{*}P\right)^t / (n -1)! .
 $$ 
 Therefore, putting the above two inequalities together, we get
 \begin{eqnarray*}
 & & \qquad  \qquad  \prod_{i=2}^{s}  \left(\frac{\log n} {\log \log n}  \right)^{-3}A_{i}  \leq \\ \nonumber
  & & \frac{\left( (s - 1)!\right)^2} { 2^{s-2} n^{s-1}} (4\pi \,  n (n-1) (n-2)) ^{s - 1}   \Delta \left( \log \Delta\right)^{n-1-q} \left(n - 1 + \log \Delta \right)^q \times \\ \nonumber
  &&\times \left(n \log^{*}P\right)^t / (n -1)!.
  \end{eqnarray*}
  From the definition of 
  $\Delta$ in \eqref{Delta}, we have
   \begin{eqnarray}\label{pAi}
&&    \prod_{i=2}^{s} A_{i}  \leq \\ \nonumber
&&   \frac{2 \left(2 \pi \,  n (n-1) (n-2)\right)^{s-1}}{(n-1)!}    \left((s -1)!\right)^{2} |D| \left(\log |D|\right)^n (\log m)^{s}  \left(\frac{\log n} {\log \log n}  \right)^{3(s-1)}.
    \end{eqnarray}
    
\subsection{Estimating $A_{1}$}\label{subA1}

To apply Proposition \ref{mat}, we  take 
  $$
  A_{1} = \max \left( d   h\left(\frac{\alpha_n-\alpha_i}{\alpha_n-\alpha_j}\right), \left|\log\left(\frac{\alpha_n-\alpha_i}{\alpha_n-\alpha_j}\right)\right|\right).
  $$
  First we estimate $h\left(\frac{\alpha_n-\alpha_i}{\alpha_n-\alpha_j}\right)$.
  For every $k \in \{1, \ldots, n\}$, let $S_{k}$ consist of the Archimedean places of $\mathbb{Q}(\alpha_{k})$ and of all non-Archimedean places corresponding to the prime ideals of $\mathbb{Q}(\alpha_{k})$ dividing $\prod_{p\leq m} p$. Let $(x , y) \in \mathbb{Z}^2$ be a non-zero solution of $|F(x , y)| \leq m$ where $a_{0} = F(1, 0)$.  We have $x - \alpha_{k} y$ is an $S_{k}$-unit of $\mathbb{Q}(\alpha_{k})$ for all solutions $(x , y)$ of $|F(x , y)| \leq m$ under consideration.  Also for $k_{1}, k_{2} \in \{1, \ldots, n\}$, we have $|S_{k_{1}}| =  |S_{k_{2}}|$. Without loss of generality, we work in  $\mathbb{K} =\mathbb{Q}(\alpha_{1})$ and let $S = S_{1}$. The same results will be valid  in every number field $\mathbb{Q}(\alpha_{k})$ with corresponding set of places $S_{k}$.

  On the one hand, $a_{0} \alpha_{1}$ is an algebraic integer. 
  On the other hand, $F(x , y)/(x-\alpha_{1}y) =  a_{0} (x - \alpha_{2} y) \ldots (x - \alpha_{n}y)$ is an algebraic integer. The ring of integers of $\mathbb{K}$ consists precisely of those elements $\gamma$ of $\mathbb{K}$ such that $|\gamma|_{v} \leq 1$ for every non-Archimedean place $v$ of $\mathbb{K}$. So for each non-Archimedean place $v$ of $\mathbb{K}$ we have
  $$
  |F(x , y)|_{v} \leq |x - \alpha_{1}|_{v} \leq |a_{0}|^{-1},
  $$
  hence
  $$
  \left| \log|x - \alpha_{1}|_{v} \right| \leq \max\left( |a_{0}|^{-1}_{v}, |F(x , y)|^{-1}_{v}  \right) \leq |a_{0}F(x , y)|^{-1}_{v}.
  $$
As a consequence, $|x - \alpha_{1} y|_{v} = 1$ for $v \not\in S$, i.e.,  $x - \alpha_{1} y$ is an $S$-unit, and
\begin{equation}\label{Referee3}
\sum_{v\in S\setminus S^{\infty}} \left| \log|x - \alpha_{1} y|_{v}\right| \leq \log|a_{0}F(x , y)|.
\end{equation}

\begin{lemma}\label{EstimateOfHeightDelta}
Assume that $|a_{0}| = |F(1 , 0)| \leq m$. 
Let $(x,y) \in \mathbb{Z}^2$ be another pair with  $|F(x , y)| \leq m$, $y>0$ and $\| \Phi(x , y)\| \geq \| \Phi(1 , 0)\|$. Then 
\begin{equation*}
h\left(\frac{\alpha_1-\alpha_i}{\alpha_1-\alpha_j}\right) \le 2\log 2 + \frac{4}{\sqrt{n}}\|\Phi(x,y)\| + 4 \log m .
\end{equation*}
\end{lemma}
\begin{proof}
Let $\beta_i = x-y \alpha_i$.
We have
\begin{equation*}
	\frac{\alpha_1-\alpha_i}{\alpha_1-\alpha_j}
	=
	\frac{\beta_1-\beta_i}{\beta_1-\beta_j}.
\end{equation*}
Thus, from the properties of the absolute logarithmic  height,
\begin{equation}\label{ok26}
	h\left(\frac{\alpha_1-\alpha_i}{\alpha_1-\alpha_j}\right)
	\le 2\log 2 + 4h(\beta_1).
\end{equation}
Let $\phi_{i}$ be as in \eqref{fi}.
Put $\mathfrak{v}_k : = \log|\beta_1|_{v_k} - \log |a_{0}/F(x , y)|^{d_{k}/n}$ for $k=1,\ldots, r_{1} + r_{2}$. 
Then $\mathfrak{v}_k  = (d_{k}/n) \left( \phi_{k}(x , y) - \phi_{k}(1 , 0)\right)$ for $k=1,\ldots, r_{1} + r_{2}$. Define
$$\bold{v}: = (\mathfrak{v}_1,\mathfrak{v}_2,\ldots,\mathfrak{v}_{r_{1}+r_{2}}, 0, \ldots, 0) \in \mathbb{R}^{s}$$  where $s = |S|$, 
and
\begin{eqnarray}\label{w}
&\, \, &\, \,  \qquad \bold{w}:  = \\ \nonumber
& & \left(\frac{d_{1}}{n}\log \left| \frac{a_{0}}{F(x , y)}\right|, \ldots,\frac{d_{r_{1}+r_{2}}}{n} \log \left| \frac{a_{0}}{F(x , y)}\right|, \log \left| \beta_{1}\right|_{v_{r_{1}+r_{2}+1}}, \ldots, \log \left| \beta_{1} \right|_{v_s}\right).
\end{eqnarray}
In the above definitions $v_{i}$ ($i= r_{1}+r_{2}, \ldots, s$) denote the non-Archimedean places in $S$, which are defined in \eqref{Refereevaluation}.
Recall that $F(1 , 0) = a_{0}$. Let $\|.\|_{1}$ denote the sum norm. Then 
\begin{equation*}
	h(\beta_1) =  \frac{1}{2}\sum_{v \in S} \left| \log|\beta_{1}|_{v} \right| \leq  \frac{1}{2}\left( \| \bold{v} \|_{1}+   \| \bold{w} \|_{1}  \right).
\end{equation*}
By Lemma \ref{lem100} we have
\begin{equation*}
	\|\bold{v}\|_{1} \le \frac{1}{n} \left(\|\Phi(x,y)\| _{1}+ \|\Phi(1,0)\|_{1} \right)\leq \frac{2}{\sqrt{n}} \|\Phi(x , y)\|.
\end{equation*}
Moreover, by \eqref{Referee3}, 
\begin{equation}\label{sizeofw}
\|\bold{w}\|_{1} \leq  \left| \log\left|\frac{a_{0}}{F(x , y)}\right| \right|  + \log |a_{0}F(x , y)| \leq 2 \log \max(|a_{0}|, |F(x , y)|)  \leq 2 \, \log m .
\end{equation}
This leads to 
$$
h(\beta_{1}) \leq \frac{1}{\sqrt{n}} \| \Phi(x , y)\| + \log m.
$$
This, together with (\ref{ok26}), completes the proof.
\end{proof}

If $\log\left(\frac{\alpha_1-\alpha_i}{\alpha_1-\alpha_j}\right) $ has its principal value, by \eqref{defoflogho1}, we have
$$\left|\log \left(\frac{\alpha_1-\alpha_i}{\alpha_1-\alpha_j}\right)\right| \leq 2d h\left(\frac{\alpha_1-\alpha_i}{\alpha_1-\alpha_j}\right) + \pi.
$$
Therefore, by Theorem \ref{EstimateOfHeightDelta}, we may take
$$
A_{1}  = \pi + 2d\left(2\log 2 + \frac{4}{\sqrt{n}}\|\Phi(x,y)\| + 4 \log m\right) < d \pi \left(\frac{4}{\sqrt{n}}\|\Phi(x,y)\| + 4 \log m\right).
$$
Combining this estimate for $A_{1}$ with \eqref{pAi}, we conclude that $\prod_{i=1}^{n} A_{i}$ is less than
\begin{eqnarray*}
 \frac{ \left(2 \pi \,  n (n-1) (n-2)\right)^{s}}{(n-1)!}    \left((s -1)!\right)^{2} |D| \left(\log |D|\right)^n (\log m)^{s}
 \times\\
\times \left[ \frac{4}{\sqrt{n}}\|\Phi(x,y)\| + 4 \log m \right]  \left(\frac{\log n} {\log \log n}  \right)^{3(s-1)}.
\end{eqnarray*}
This confirms the value for 
$\Omega$  in \eqref{Om}.

\subsection{Estimating $B$}\label{subB}

Let $S =\{v_{1}, \ldots, v_{s}\}$, where $v_{i}$'s are defined in \eqref{Refereevaluation}. Let
$$\mathbf{b} = (b_{1}, \ldots, b_{s -1}),$$
 where $b_{k}$'s are the coefficients of logarithms in \eqref{bkinrep}.
Define the matrix 
$$
\mathbf{E} := \left( \log |\epsilon_{i}|_{v_{j}}\ \right)_{i = 1, \ldots s -1}.
$$
    We have the following matrix multiplication.
      $$
   \mathbf{b}  \mathbf{E} =  (\log|\beta_{1}|_{v_{1}}, \ldots \log|\beta_{1}|_{v_{s}}),
      $$
      where $\beta_{1}  = x_{3} - \alpha_{1} y_{3}$.
  We have 
        $$
 \|  \mathbf{b}  \mathbf{E}\|^2 \leq 2 \|\Phi(x_{3} , y_{3}) -  \Phi(1 , 0)\|^2 + \| \bold{w}\|^2,
   $$
   where $\mathbf{w}$ is defined in \eqref{w}. Then by part (3) of Lemma \ref{L1BugG}, we deduce that
   \begin{eqnarray*}
   \max |b_{i}| &\leq& \sqrt{2} \left[ (s-1)! \right]^2 \frac{53\, n \log 6n}{2^{s-2}}\left(\sqrt{2} \|\Phi(x_{3} , y_{3}) -  \Phi(1 , 0)\| +  \|\bold{w}\| \right)\\
   & \leq &   \sqrt{2} \left[ (s-1)! \right]^2 \frac{53\, n \log 6 n}{2^{s-2}}  \left(2 \sqrt{2}\mathbf{r}_{3} + 2\,  \log m   \right),
   \end{eqnarray*}
   where the last inequality is deduced from \eqref{sizeofw}. When the function $\log$ has its principle value then in \eqref{bkinrep}, we have
   $$
   2 w \leq \sum_{i= 1}^{s-1} |b_{i}| \leq (s-1)  \max |b_{i}|.
   $$
   This leads us to the following choice for $B$ in Proposition \ref{mat}:
   $$
B = 4 (s-1) \, \left[ (s-1)! \right]^2 \frac{53\, n \log 6 n}{2^{s-2}}  \left( \mathbf{r}_{3} +  \log m   \right).
$$
This establishes  the value for  $W_{0}$ in \eqref{W00}.

  \subsection{Completing the proofs by contradiction}

  We are going to combine  the gap principles that are established in Theorems  \eqref{exg} and \eqref{exg3,4},  with Proposition \ref{mat} to give an upper bound for the number of possible large solutions. 
  
\textbf{Remark}.  If we start with 3 solutions, the gap principle  works, but the constants from Proposition \ref{mat} are too large to provide a contradiction. We should remark that we do not believe these constants are sharp. However, with $5$ solutions and applying the gap principle twice we get a contradiction. This will lead us to conclude that there are at most $4$ large solutions.

 First assume that  the degree of the binary form $F(x , y)$ is greater than $4$ and there are $5$  solutions $(x_{1}, y_{1})$, $(x_{2}, y_{2})$, $(x_{3}, y_{3})$, $(x_{4} , y_{4})$, $(x_{5} , y_{5})$ to $|F (x , y) | \leq m$, satisfying the following conditions
 $$
 y_{l} > C\frac{2}{\sqrt{3}} (n+1)^{n}  \sqrt{n} \,  m^{n}  |D|^{\frac{1}{n^2 (n-2)}} M(F)^{\frac{2n -2}{n(n-2)} + n-1}
 $$
and
 $$
\left| x_{l} - \alpha_{n} y_{l}\right|  = \min_{1\leq i \leq n} \left| x_{l} - \alpha_{i} y_{l}\right|,   \quad l \in \{1 , 2, 3, 4, 5 \},
$$
where $\alpha_{n}$ is a real root of $F(X , 1) =0$ and $C$ is a positive number to be specified later.
Assume that $\mathbf{r}_{1} \leq \mathbf{r}_{2} \leq \mathbf{r}_{3} \leq \mathbf{r}_{4} \leq \mathbf{r}_{5} $, where $\mathbf{r}_{j} = \left\| \Phi( x_{j} , y_{j}) \right\|$.
Let 
\begin{eqnarray}\label{vk}
K =  2500  (s+ 2)^{2s + 7}  (4\pi)^{s+2}    \left(  n (n-1) (n-2)\right)^{s+2}  |D| \left(\log |D|\right)^n \times
\\ \nonumber
\times (\log m)^{s+1}\,   \log \log m \left(\frac{\log n}{\log \log n} \right)^{3(s-1)}.
\end{eqnarray}

Inequality \eqref{almightyEs} implies that 
$$
\log \left|  \log \left(\frac{(\alpha_{n} - \alpha_{i}) (x_{3} - \alpha_{j}y_{3})}{(\alpha_{n} -\alpha_{j})(x_{3} - \alpha_{i}y_{3})} \right) \right|>  - K  \bold{r}_{1} \log \bold{r}_{3}.
$$
Similarly,
$$
\log \left|  \log \left(\frac{(\alpha_{n} - \alpha_{i}) (x_{5} - \alpha_{j}y_{5})}{(\alpha_{n} -\alpha_{j})(x_{5} - \alpha_{i}y_{5})} \right) \right|  > - K  \bold{r}_{3} \log \bold{r}_{5},
$$
where  $(i, j), (i', j') \in \{1, \ldots , n-1\} \times \{1,\ldots , n-1\}$ are chosen according to Lemma \ref{Tu}.
Comparing this with Lemma \ref{Tu+pi}, we have
$$
-\log C +\log\left( \sqrt{ \frac{n}{n-2}}\right)+ \frac{-\bold{r}_{3}}{n\sqrt{n}}> - K   \bold{r}_{1}\log \bold{r}_{3}.
$$
 By Lemma \ref{Dr}, the value $\bold{r}_{3}$ is large enough to satisfy
$$
 \bold{r}_{3}^{\frac{e-1}{e}} <  \frac{\bold{r}_{3}}{\log\bold{r}_{3}},
$$
where $e = \exp(1)$.
So for the constant 
\begin{equation}\label{K1def}
K_{1} = \left(n\sqrt{n} K\right)^{\frac{e}{e-1}},
\end{equation}
 we have
$$
\bold{r}_{3} < K_{1} \bold{r}_{1}^{\frac{e}{e-1}}.
$$
Similarly,
$$
\bold{r}_{5} < K_{1} \bold{r}_{3}^{\frac{e}{e-1}}.
$$
By Theorem  \ref{exg}, we have  
\begin{eqnarray*}
\mathbf{r}_{3} >A\,  \exp\left( \frac{ \mathbf{r}_{1}}{n\sqrt{n}}\right) \, \, \textrm{and} \, \, \mathbf{r}_{5} > A\,  \exp\left( \frac{ \mathbf{r}_{3}}{n\sqrt{n}}\right),
\end{eqnarray*}
where 
$$
A = \frac{C}{4}\, \left((n-1)^2 + n-1\right)^{3/4} \left[(n-2) (n-1) \, \log m +   (n-2)^2 \log \left(  M(F) \right)\right]^{3/2}.$$
Therefore,
\begin{eqnarray}\label{contradictionineq}
A\,  \exp\left( \frac{ \mathbf{r}_{3}}{n\sqrt{n}}\right) &<&   \bold{r}_{5} <  K_{1} \bold{r}_{3}^{\frac{e}{e-1}}\\ \nonumber
& <&  K_{1} \left[ K_{1} \bold{r}_{1}^{\frac{e}{e-1}} \right]^{\frac{e}{e-1}} ,
\end{eqnarray}
which is a contradiction, as $\mathbf{r}_{3} >A\,  \exp\left( \frac{ \mathbf{r}_{1}}{n\sqrt{n}}\right)$.

\begin{lemma}\label{large=}
Assume that $F$ has degree $n \geq 5$ and that 
$$
m < \left(\frac{2}{7}  \right)^n \left( \frac{|D|}{n^n}\right)^{\frac{1}{2(n-1)}}.
$$
Then the equation  $|F(x , y)| \leq m$ has at most 
 $4 (n -2q)$ solutions $(x , y) \in \mathbb{Z}^{2}$ with $y \geq M(F)^{1+(n-1)^2}$.
 \end{lemma}
\begin{proof}
We assume that there is a real root of $F(X, 1)$, say $\alpha_{n}$,  so that there are $5$ solutions $(x , y) \in \mathbb{Z}^{2}$ with $y \geq M(F)^{1+(n-1)^2}$ that are related to $\alpha_{n}$. We apply Theorem  \ref{exg} with 
$$
C = \frac{M(F)^{1+(n-1)^2}}{\frac{2}{\sqrt{3}} (n+1)^{n}  \sqrt{n} \,  m^{n}  |D|^{\frac{1}{n^2 (n-2)}} M(F)^{\frac{2n -2}{n(n-2)} + n-1}}.
$$
Then \eqref{contradictionineq} provides a contradiction. 
We conclude that there can exist  at most $4$ primitive solutions $(x , y)$ with $y \geq M(F)^{1+(n-1)^2}$ that are related to any fixed real root of $F(X, 1)$. The proof is complete since the number of real roots of $F(X, 1)$ is $n -2q$.
\end{proof}

Now assume that the degree $n \geq 3$ and that there are $7$ primitive  solutions with $y \geq M(F)^{1+(n-1)^2}$ that are related to $\alpha_{n}$. 
By Theorem \ref{exg3,4}, we have  
\begin{eqnarray*}
\mathbf{r}_{4} >A\,  \exp\left( \frac{ \mathbf{r}_{1}}{n\sqrt{n}}\right) \, \, \textrm{and} \, \, \mathbf{r}_{7} > A\,  \exp\left( \frac{ \mathbf{r}_{4}}{n\sqrt{n}}\right),
\end{eqnarray*}
where 
$$A = \frac{C}{8 \sqrt{2}}\, \left((n-1)^2 + n-1\right)^{3/4} \left[(n-2) n \, \log m +  (n-1) (n-2)\log \left(  M(F) \right)\right]^{3/2}.
$$
Therefore,
\begin{eqnarray}\label{contradictionineq3,4}
A\,  \exp\left( \frac{ \mathbf{r}_{4}}{n\sqrt{n}}\right) &<&   \bold{r}_{7} <  K_{1} \bold{r}_{4}^{\frac{e}{e-1}}\\ \nonumber
& <&  K_{1} \left[ K_{1} \bold{r}_{1}^{\frac{e}{e-1}} \right]^{\frac{e}{e-1}} ,
\end{eqnarray}
which is a contradiction, as $\mathbf{r}_{4} >A\,  \exp\left( \frac{ \mathbf{r}_{1}}{n\sqrt{n}}\right)$.

\begin{lemma}\label{large=3,4}
Assume that $F$ has degree $n \geq 3$ and that 
$$
m < \left(\frac{2}{7}  \right)^n \left( \frac{|D|}{n^n}\right)^{\frac{1}{2(n-1)}}.
$$
Then the equation
$|F(x , y)| \leq m$ has  at most $6 (n -2q)$ primitive solutions $(x , y)\in \mathbb{Z}^2$ with $y \geq M(F)^{1+(n-1)^2}$.
\end{lemma}
\begin{proof}
We assume that there is a real root of $F(X, 1)$, say $\alpha_{n}$,  so that there are $7$ primitive solutions 
$(x , y) \in \mathbb{Z}^{2}$ with $y \geq M(F)^{1+(n-1)^2}$ that are related to $\alpha_{n}$. We apply Theorem  \ref{exg3,4} with
$$
C = \frac{M(F)^{1+(n-1)^2}}{\frac{2}{\sqrt{3}} (n+1)^{n}  \sqrt{n} \,  m^{n}  |D|^{\frac{1}{n^2 (n-2)}} M(F)^{\frac{2n -2}{n(n-2)} + n-1}}.
$$
Then \eqref{contradictionineq3,4} provides a contradiction. 
We conclude that there can exist  at most $6$ primitive solutions $(x , y)$ with $y \geq M(F)^{1+(n-1)^2}$ that are related to any fixed real root of $F(X, 1)$. The proof is complete  since the number of real roots of $F(X, 1)$ is $n -2q$.
\end{proof}

Now the proofs of Theorems \ref{mainineq} and \ref{maineq}  are completed by combining 
Lemma \ref{large=}
with Lemmata 
 \ref{smallineq}, \ref{smallequ}, \ref{SC1}--\ref{SC1G}.

Finally we will consider the equation $|F(x , y)| = m$, with no assumption on the size of  $m$ in terms of the discriminant of $F$.
 In this case, 
$|S| \leq n + n \pi (m)$. 
In Theorem \ref{exg}, we take $C = m^{n} M(F)^{(n-1)^2-n} $. Assuming that there are $5$ solutions with $$y \geq \left(2^{n/(n-2)} n^{\frac{2n -1}{2n-4}} m^{\frac{1}{n-2}} M(F)\right)^{1+(n-1)^2}
$$ when $n\geq 5$, the application of Theorem \ref{exg} twice provides  a contradiction  that leads us to the following Lemma.
 \begin{lemma}\label{largegeneral}
 Let $F(x , y)$ be an irreducible binary form of degree $n \geq 5$.
The equation
$|F(x , y)| = m$ has at most $4 (n -2q)$ primitive solutions $(x , y)$ with 
$$y \geq \left(2^{n/(n-2)} n^{\frac{2n -1}{2n-4}} m^{\frac{1}{n-2}} M(F)\right)^{1+(n-1)^2}.
$$
\end{lemma}
Similarly, Theorem \ref{exg3,4} implies that
\begin{lemma}\label{largegeneral3,4}
 Let $F(x , y)$ be an irreducible binary form of degree $n \geq 3$.
The equation
$|F(x , y)| = m$ has at most $6 (n -2q)$ primitive solutions $(x , y)$ with 
$$y \geq \left(2^{n/(n-2)} n^{\frac{2n -1}{2n-4}} m^{\frac{1}{n-2}} M(F)\right)^{1+(n-1)^2}.
$$
\end{lemma}

Lemmata \ref{largegeneral} and \ref{largegeneral3,4}, together with Lemma \ref{SC1G}, complete the proof of Theorem \ref{G}.

\section{Acknowledgments}

I am extremely grateful to the anonymous referee for reading this manuscript carefully and providing plenty of insightful comments and suggestions.  I would also like to thank  Professor Jeffrey Vaaler for answering my questions about the heights of algebraic numbers.   
Part of this work has been done during my visit to  Hausdorff Research Institute for  Mathematics in Bonn, Germany. I would like to thank the institute staff, as well as the organizers of the Arithmetic and Geometry program in winter 2013, for giving me the opportunity to participate in this program.

\end{document}